\documentclass[11pt]{article}
\title{Edge Element Approximation for the Spherical Interface Dynamo System}
\author{ Junqing Chen\thanks{\footnotesize
		Department of Mathematical Sciences, Tsinghua University, Beijing
		100084, China. The work of this author was supported by National Key  R\&D Program of China 2019YFA0709600, 2019YFA0709602. (jqchen@tsinghua.edu.cn).}
		\and Ming Sun\thanks{\footnotesize Department of Mathematical Sciences, Tsinghua University, Beijing 100084, China. (sun-m21@mails.tsinghua.edu.cn).} }
\date{}

\usepackage{amssymb,amsmath,amsthm}
\usepackage{graphicx,subfig,color}
\usepackage{cases}
\usepackage{exscale}
\usepackage{relsize}
\usepackage{float}
\usepackage{caption}
\usepackage{array}
\usepackage{booktabs}
\usepackage{makecell}
\usepackage{hyperref}

\newcommand{\bG}{{\bf G}}

\newcommand{\bU}{{\bf U}}
\newcommand{\bV}{{\bf V}}
\newcommand{\curl}{{\bf curl}}

\newcommand{\bdL}{{\boldsymbol{L}}}
\newcommand{\bdF}{{\boldsymbol{F}}}
\newcommand{\bdA}{{\boldsymbol{A}}}
\newcommand{\bdB}{{\boldsymbol{B}}}
\newcommand{\bdu}{{\boldsymbol{u}}}
\newcommand{\bdv}{{\boldsymbol{v}}}
\newcommand{\bdx}{{\boldsymbol{x}}}
\newcommand{\bdw}{{\boldsymbol{w}}}
\newcommand{\bdn}{{\boldsymbol{n}}}

 \newtheorem{theorem}{Theorem}[section] %
\newtheorem{lemma}[theorem]{Lemma} %
\newtheorem{remark}{Remark}[section] %
\numberwithin{equation}{section}

\begin{document}
	
	\maketitle
	
\begin{abstract}
Exploring the origin and properties of magnetic fields is crucial to the development of many fields such as physics, astronomy and meteorology. We focus on the edge element approximation and theoretical analysis of celestial dynamo system with quasi-vacuum boundary conditions. The system not only ensures that the magnetic field on the spherical shell is generated from the dynamo model, but also provides convenience for the application of the edge element method. We demonstrate the existence, uniqueness and stability of the solution to the system by the fixed point theorem. Then, we approximate the system using the edge element method, which is more efficient in dealing with electromagnetic field problems. Moreover, we also discuss the stability of the corresponding discrete scheme. And the convergence is demonstrated by later numerical tests. Finally, we simulate the three-dimensional time evolution of the spherical interface dynamo model, and the characteristics of the simulated magnetic field are consistent with existing work.		
\end{abstract}
\vskip 2ex\par\noindent\normalfont{\bfseries Keywords: }celestial magnetic fields, interface dynamo system, quasi-vacuum boundary conditions, edge element method

\par\noindent\normalfont{\bfseries Mathematics Subject Classification}(MSC2020): 65M60, 65M12

\section{Introduction}
It is known that magnetic fields exist in many celestial bodies such as Earth and Sun\cite{Kono2002,Zhang2003,Chan2006}. The magnetic fields are very important for our  living environment.  Therefore it is of  great significance to investigate the origin as well as the properties of celestial magnetic fields.

Until now, the dynamo theory is one of the more accepted hypotheses about the origin of celestial magnetic fields.
In 1919, Larmor first proposed the concept of the solar dynamo \cite{Larmor1919}.
Later, Elsasser proposed basic mathematical models for the terrestrial magnetic fields \cite{Elsasser1946a,Elsasser1946} and Parker developed the dynamo theory for the solar magnetics fields \cite{Parker1955}.
Steenbeck et al. provided  more detailed mathematical description of the solar dynamo model, proposing the mean-field dynamo model and introducing the concept of the $\alpha$-effect (i.e., small-scale magnetic field and velocity field interacting to produce the large-scale magnetic field) in 1966 \cite{Kono2002, Steenbeck1966}.
The model avoids the anti-generator theory, but causes the $\alpha$-quenching effect. 
In 1993 , Parker proposed an interface dynamo model to describe the location and essence of the $\alpha$-effect, thus avoided the $\alpha$-quenching problem \cite{Parker1993a}.
In recent years, Zhang et al. developed a multilayered interface dynamo model \cite{Zhang2003,Chan2008}.  In that model,  the domain is divided into four regions and the generation mechanism of the magnetic field is presented. For more details of the dynamo theory, we suggest \cite{Hollerbach1996,Ossendrijver2003,Brandenburg2005,Jones2011} for reference.

The numerical simulation of the dynamo model on the three-dimensional sphere is very difficult due to its complexity as well as nonlinearity.
Currently, the popular methods are the spectral method using spherical harmonic functions and the finite element method \cite{Zhang1989,Glatzmaier1995,Chan2006,Chan2008,Cheng2020}.
However, the global nature of the spectral method and the slow Legendre transform make the method inefficient in dealing with the dynamo model with space-time variable parameters \cite{Chan2001,Chan2008}.
As for the finite element method, Zou, Zhang and their co-workers provided the mathematical theory of spherical interface dynamo as well as finite element approximation framework \cite{Chan2006,Zhang2003}.
However, to deal with the solenoidal condition, the model is recast to  a saddle point equation which is difficult to solve efficiently\cite{Zhang2003,Cheng2020}. 

It is well known that the N\'ed\'elec edge element method is very effective for solving electromagnetic field problems \cite{Nedelec1980,Nedelec1986,Biro1996,Zhong2012}.
As we will see in the next section, the spherical dynamo system is a nonlinear evolution system concerning the curl operator. Then it is nature to use the edge element to discretize the magnetic field.  However, to use the edge element method, we need suitable boundary and interface conditions. 
Fortunately, Gilman et al. proposed that quasi-vacuum boundary conditions are applicable to the solar dynamo problem because the observed magnetic field at the solar layer is always perpendicular to the surface \cite{Kageyama1997,Harder2005}. The quasi-vacuum condition allows the magnetic field to have only the radial component at the outer boundary.  It is just the homogeneous Dirichlet boundary condition in tangential directions. We further relax the interface conditions and only keep the tangential continuity. 
Finally, we have a new spherical interface dynamo model which is suitable for edge element approximation. With this model, it is possible for us to solve the spherical interface problem efficiently with fast algorithm such as \cite{Hiptmair2007}.  

We emphasize  that the new model is different with the one in \cite{Chan2006}  and it does not contain the divergence-free equation anymore. Then the existence and uniqueness of solution can not be proven by Galerkin methods since the corresponding function space is lack of necessary compact embedding property.  So in the present work, we prove the well-posedness of the new model by fixed point method.  We also propose a stable numerical scheme based on edge element discretization and confirm that the numerical solution of the new model is consistent with existing work.

The rest of the paper is organized as follows.
In Section \ref{sect2}, we introduce the interface dynamo system in detail and provides the theorems used later.
In Section \ref{sect3}, we discuss the weak formulation and well-posedness of the system by fixed point theorem.
We construct a fully discrete numerical scheme and present the stability, existence and uniqueness of the scheme in Section \ref{sect4}.
In Section \ref{sect5}  we show some numerical experiments to demonstrates the effectiveness of our method. Finally in Section \ref{sect6}, we draw some conclusion.

\section{Spherical interface dynamo system}\label{sect2}
In this section, we mainly introduce the spherical interface dynamo model with the quasi-vacuum boundary condition.
Meanwhile, some necessary preliminaries are provided at the end of this section for analysis in next sections.

\subsection{The model}

\begin{figure}[!htbp]
	\centering
	\includegraphics[width=0.6\textwidth]{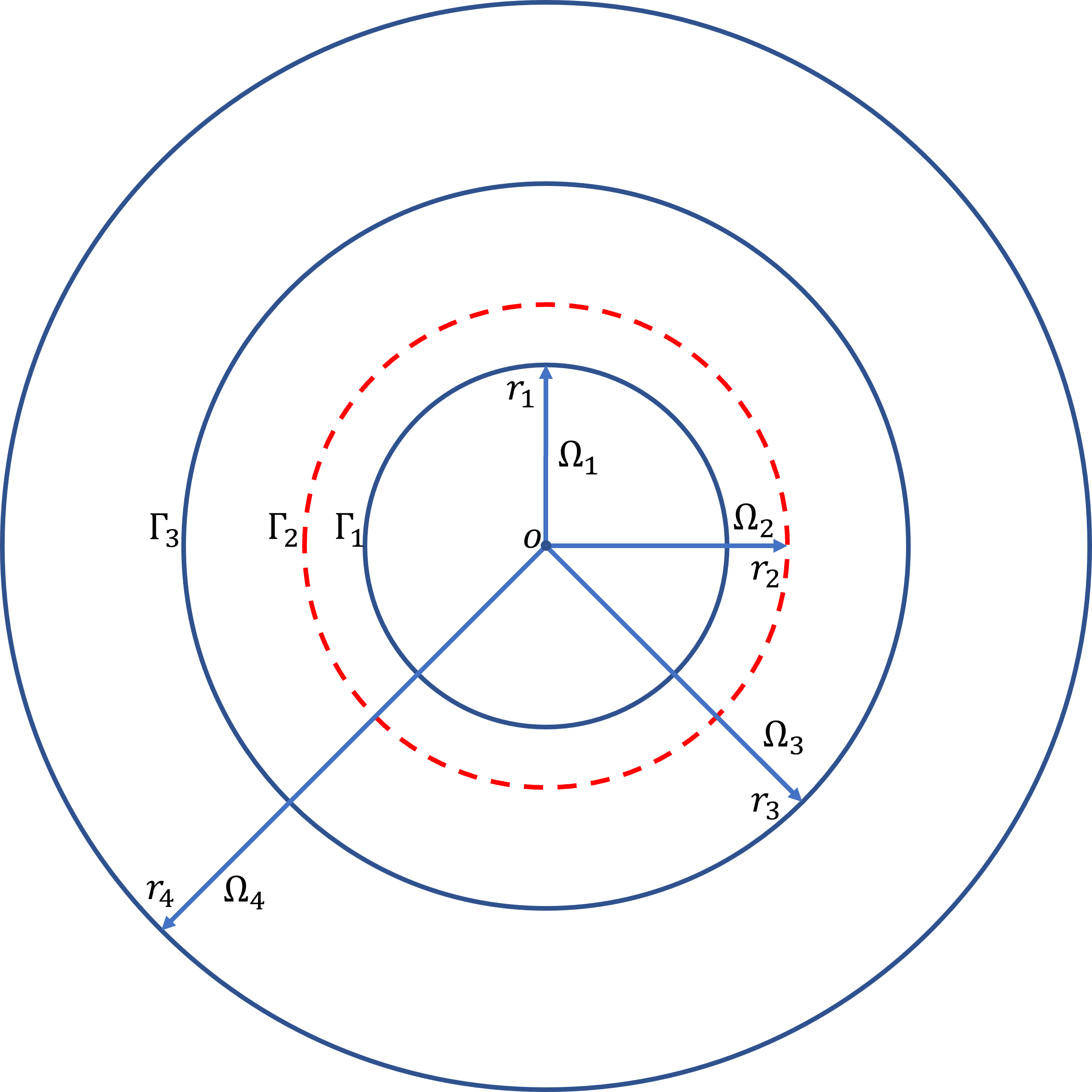}
	\caption{Physical domain of the interface dynamo model.   \label{fig:Domain} }
\end{figure}

We consider the interface dynamo model with the fully turbulent convection zone based on the assumptions in \cite{Zhang2003}.
According to the different mechanisms of magnetic field generation, the physical domain $\Omega$ of the model is divided into four regions $\Omega_i (i=1,2,3,4)$ as shown in Figure~\ref{fig:Domain}. We denote by $\bdB$ the magnetic field and $\bdB_i$ the magnetic field in $\Omega_i$ for $i=1,2,3,4$. 

The region $\Omega_3 (r_2 < r < r_3)$ represents the convection zone with the magnetic diffusivity $\beta_3(\bdx)$, where $r$ denotes the radius of the spherical domain.
Under the $\alpha$-effect, a weak magnetic field $\bdB_3$ is generated in this region, which is governed by 
\begin{equation}
	\frac{\partial \bdB_3}{\partial t}+\nabla \times(\beta_3(\bdx) \nabla \times \bdB_3)=R_{\alpha} \nabla \times\left(\frac{f(\bdx, t)}{1+\sigma|\bdB_3|^{2}} \bdB_3\right) \quad \text{ in } \quad \Omega_3 \times(0, T).       
	\label{B3}
\end{equation}
The right hand side nonlinear term $\frac{R_{\alpha} f(\bdx, t)}{1+\sigma|\bdB_3|^{2}}$ denotes the local $\alpha$-quenching effect, where $f(\bdx, t)$ is a model-oriented known function, $R_\alpha$ is the magnetic $\alpha$ Reynolds number, and $\sigma$ is a constant parameter. For details about the $\alpha$-effect, refer to \cite{Brandenburg1994,Choudhuri1995,Kuker2001} .

The external insulating region $\Omega_4 (r_3 < r < r_4)$ with a large magnetic diffusivity $\beta_4(\bdx)$ is outside the convection zone, which is almost electrically insulated and the magnetic field would not be generated.
The diffusion of the magnetic field $\bdB_4$ in this region is described by
\begin{equation}
	\frac{\partial \bdB_4}{\partial t}+\nabla \times(\beta_4(\bdx) \nabla \times \bdB_4)=0 \quad  \text{ in } \quad \Omega_4 \times (0, T).         
	\label{B4}
\end{equation}

Below the convection zone is the tachocline $\Omega_2 (r_1 < r < r_2)$ with strong differential rotation, which has the effect of enhancing the magnetic field.
In this region, the differential rotation shears the weak radial magnetic field generated by the convection zone and diffused into the tachocline, yielding the strong magnetic field $\bdB_2$ governed by
\begin{equation}
	\frac{\partial \bdB_2}{\partial t}+\nabla \times(\beta_2(\bdx) \nabla \times \bdB_2)=R_{m} \nabla \times(\bdu(\bdx, t) \times \bdB_2) \quad \text{ in } \quad \Omega_2 \times(0, T).        
	\label{B2}
\end{equation}
The parameters $R_m$ and $\beta_2(\bdx)$ correspond to the magnetic Reynolds number and magnetic diffusivity in the tachocline, respectively.
And the vector function $\bdu(\bdx, t)$ represents small scale turbulence.

The internal radiation zone $\Omega_1 (0 < r < r_1)$ is a uniformly rotating sphere, and the magnetic field cannot be generated in this zone. The diffusion process of the magnetic field $\bdB_1$ is governed by
\begin{equation}
	\frac{\partial \bdB_1}{\partial t}+\nabla \times(\beta_1(\bdx) \nabla \times \bdB_1)=0 \quad  \text{ in } \quad \Omega_1 \times (0, T),         
	\label{B1}
\end{equation}
where $\beta_1(\bdx)$ is the magnetic diffusivity.

We close the above system by applying initial and boundary conditions as follows
\begin{equation}
	\bdB(\bdx, 0)=\bdB_{0}(\bdx) \quad \text{ in } \quad \Omega,
	\label{CB1}
\end{equation}
\begin{equation}
	\bdB \times  \bdn = 0 \quad \text{ on } \quad \partial\Omega \times (0, T),
	\label{CB2}
\end{equation}
where $\bdn$ denotes the unit outward normal vector to the boundary $\partial\Omega$.
Here, \eqref{CB2} is just the quasi-vacuum boundary condition. Let $\Gamma_i (i=1,2,3)$ denote the interfaces between regions in the above system, respectively, as shown in Figure~\ref{fig:Domain}.
With the combination of the jump of magnetic diffusivity $\beta(\bdx)$ at interfaces, we impose physical jump conditions as follows 
\begin{equation}
	\left[\left(\beta(\bdx) \nabla \times \bdB\right) \times \bdn\right] =0, \quad \left[\bdB \times  \bdn\right] = 0 \quad \text{ on } \quad \left(\Gamma_1 \cup \Gamma_2 \cup \Gamma_3 \right) \times (0, T),
	\label{TY1}
\end{equation}
where $\left[A\right]$ means the jumps of $A$ across the interfaces.

Combining the system \eqref{B3}-\eqref{B1} in regions $\Omega_i (i=1,2,3,4)$, the initial and boundary conditions \eqref{CB1}-\eqref{CB2} and the physical jump conditions \eqref{TY1}, the spherical interface dynamo model can be written as follows
\begin{equation}\label{Z1}
	\begin{aligned}
		&\frac{\partial \bdB}{\partial t}+\nabla \times(\beta(\bdx) \nabla \times \bdB) \\
		&\qquad=R_{\alpha} \nabla \times\left(\frac{f(\bdx, t)}{1+\sigma|\bdB|^{2}} \bdB\right)+R_{m} \nabla \times(\bdu(\bdx, t) \times \bdB) \quad \text{ in } \quad \Omega \times(0, T),\\
		&\bdB(\bdx, 0)=\bdB_{0}(\bdx) \quad \text{ in } \quad \Omega,\\	
		&\bdB \times  \bdn = 0 \quad \text{ on } \quad \partial\Omega \times (0, T),
	\end{aligned}
\end{equation}
where $\beta(\bdx)$  is the piecewise smooth magnetic diffusivity and $0<\beta_0\leq\beta(\bdx)\leq \beta_M$.
By definition of the model, the function $f(\bdx, t)$ and the velocity field $\bdu(\bdx, t)$ are supported only in the convection zone $\Omega_3$ and the tachocline $\Omega_2$, respectively.  We also assume that $\bdu(\bdx, t)$ is nonslip on the boundary of the tachocline, which suggests that $f$ and $\bdu$ vanish on $\Gamma_1$, $\Gamma_2$, and $\Gamma_3$.

We remark that the system \eqref{Z1} is different with the one in \cite{Chan2006}. We omit the divergence free condition $\nabla\cdot\bdB(\bdx,t)=0, t>0$. To deal with this condition,  we assume that the initial data $\bdB_0(\bdx)$ is divergence free 
$$\nabla\cdot\bdB_0(\bdx)=0 \mbox{ in }\Omega.$$
In the next section, we will prove that, with this assumption, the solution to system \eqref{Z1} satisfies the divergence free condition automatically.


\subsection{Preliminaries}
For further analysis of the above system \eqref{Z1},  some mathematical notations as well as the auxiliary lemmas are collected in this subsection.
Let $H^1(\Omega)=\{u\in L^2(\Omega), \nabla u\in L^2(\Omega)^3\}$ be the usual Sobolev space and use $\boldsymbol{H}^1(\Omega)$ to denote $H^1(\Omega)^3$. 
Similarly, $L^2(\Omega)^3$ is denoted by $\boldsymbol{L}^2(\Omega)$.
The most frequently used spaces in this paper are the following Hilbert spaces
$$H\left(\curl;\Omega\right)=\left\{\bdA \in \boldsymbol{L}^2(\Omega);\;\;\nabla\times \bdA \in \boldsymbol{L}^2(\Omega)\right\},$$
$$H_0\left(\curl;\Omega\right)=\left\{\bdA \in \boldsymbol{L}^2(\Omega);\;\;\nabla\times \bdA \in \boldsymbol{L}^2(\Omega),\;\bdA \times \bdn = 0 \; \text{on} \; \partial\Omega\right\}$$
with graph norm
\begin{eqnarray*}
	\|\bdu\|_{H(\curl;\Omega)}=\left(\|\bdu\|^2_{\bdL^2(\Omega)}+\|\nabla\times \bdu\|^2_{\bdL^2(\Omega)}\right)^{1/2}.
\end{eqnarray*}
It is known that $H(\curl;\Omega)$ and $H_0(\curl;\Omega)$ are separable \cite{Ciarlet2018}.
Usually, the notation $\left(\cdot , \cdot\right)$ denotes the scalar product in $L^2(\Omega)$ or $\boldsymbol{L}^2(\Omega)$, while $\left\langle \cdot , \cdot \right\rangle$ is used to denote the dual pairing between any two Hilbert spaces. 
For simplification of the later discussions, we write $\bV = H_0\left(\curl;\Omega\right)$, write $\left\| \cdot \right\|^{2}_{\beta} \triangleq \left(\beta \cdot , \cdot\right)$ for a non-negative function $\beta(\bdx)$, and let $\left\| \cdot \right\|^{2} \triangleq \left(\cdot , \cdot\right)$.  We also denote the dual space of $\bV$ by $\bU$ in the following.

Subsequently, we recall some technique results to support the theoretical analysis later.
\begin{lemma}[Young's inequality with $\varepsilon$ \cite{Evans2010}]
	For any $a,b \geq 0$ and $\varepsilon > 0$, suppose $1 < p, q < \infty$ and $\frac{1}{p} + \frac{1}{q} = 1$, then
	\begin{equation*}
		a b \leq \frac{\varepsilon}{p}a^p  + \frac{\varepsilon^{-q/p}}{q}b^q.
	\end{equation*}
\end{lemma}


\begin{lemma}[Integration by parts \cite{Monk2003}] 
	Let $\Omega$ be a bounded Lipschitz domain in $\mathbb{R}^3$.
	The mapping $\gamma_t:\bdv \rightarrow \bdn \times \bdv|_{\partial\Omega}$ can be extended by continuity to a continuous linear map from $H\left(\curl;\Omega\right)$ into $(H^{-1/2}\left(\partial\Omega\right))^3$, where $\bdn$ is the unit outward normal to $\Omega$.
	Then, the following Green's theorem holds:
	$$
	\left(\nabla \times \bdv,\phi\right) - \left(\bdv,\nabla \times \phi\right) = \langle\gamma_t(\bdv),\phi\rangle_{\partial\Omega}, \qquad \forall \, \bdv \in H\left(\curl;\Omega\right), \,  \forall \, \phi \in \boldsymbol{H}^1\left(\Omega\right).
	$$
\end{lemma}

\begin{lemma}[Gronwall's inequality (differential form) \cite{Evans2010}]\label{gronwall}
	Let $g(t)$ be a nonnegative, absolutely continuous function on $[0,T]$ which satisfies the following inequality
	\begin{eqnarray*}
		g^\prime(t)\leq \phi(t) g(t)+\psi(t)\quad \mbox{ a.e.} \; t \in[0, T],
	\end{eqnarray*}
	where $\phi(t)$ and $\psi(t)$ is nonnegative summable functions on $[0, T]$. Then, we have the following estimate
	\begin{eqnarray*}
		g(t)\leq e^{\int^t_0 \phi(s) d s}\left[g(0)+\int^t_0\psi(\tau)d\tau\right]\quad \forall \; t \in[0, T].
	\end{eqnarray*}
\end{lemma}

\begin{lemma}[Discrete Gronwall's inequality \cite{Liao2010}]
	Assume that $\left\{\psi^k\right\}^{\infty}_{k=0}$ is a non-negative sequence, $\Phi^0 \geq 0$, and $\left\{F^k\right\}^{\infty}_{k=0}$ is a non-negative increasing sequence, which satisfies
	$$ 
	F^k \leq \Phi^0 +  c \tau \displaystyle\sum_{l=0}^{k-1}F^l +  c \tau \displaystyle\sum_{l=0}^{k}\psi^l, \quad k = 0,1,2,\cdots, 
	$$
	then we have
	$$ 
	F^k \leq e^{c k \tau} \left(\Phi^0 +  c \tau \displaystyle\sum_{l=0}^{k}\psi^l \right) , \quad k = 0,1,2,\cdots.
	$$	
\end{lemma}

\begin{theorem}[Banach's fixed point theorem \cite{Evans2010}]\label{fpt} Let $X$ be a Banach space with norm $\|\cdot\|$.  Assume 
	$$S: X\rightarrow X$$
	is a nonlinear mapping, and suppose that 
	$$\|S(u)-S(v)\|\leq \gamma \|u-v\|~(u,v \in X)$$
	for some constant $\gamma<1$. Then $S$ has a unique fixed point. 
\end{theorem}

\section{The weak formulation and well-posedness}\label{sect3}

In this section, we introduce the weak formulation of the spherical interface dynamo system \eqref{Z1} and consider the well-posedness of the problem.
The weak formulation is  as follows:

For any $\bdA \in \bV$, find $\bdB \in  L^{2}(0, T; \bV)$ and $\bdB^\prime\in L^2(0,T; \bU)$ such that 
\begin{equation}
	\begin{aligned}
		\left<\bdB^{\prime}, \bdA\right>+\left(\beta \nabla \times \bdB, \nabla \times \bdA\right) =  R_{\alpha}\left(\frac{f}{1+\sigma|\bdB|^{2}} \bdB, \nabla \times \bdA\right)+R_{m}\left(\bdu \times \bdB, \nabla \times \bdA\right)
	\end{aligned}
	\label{BF1}
\end{equation}
provided that $\bdB(\bdx, 0)=\bdB_{0}(\bdx)$, $\nabla\cdot\bdB_0(\bdx)=0$, $f(\bdx,t) \in H^1(0,T;L^\infty(\Omega))$ and $\bdu(\bdx,t) \in H^1(0,T;\bdL^\infty(\Omega))$.
In the rest of this section, we prove the well-posedness of the problem \eqref{BF1} step by step.
\subsection{The well-posedness of linear system}
To prove the well-posedness of the nonlinear problem, we need firstly the well-posedness of the corresponding linear system
\begin{equation}
	\begin{aligned}
		&\frac{\partial \bdB}{\partial t}+\nabla \times(\beta(\bdx) \nabla \times \bdB) =\bdF \quad \text{ in } \quad \Omega \times(0, T),\\
		&\bdn\times\bdB(\bdx,t)=0 \quad \text{ on } \quad \partial\Omega\times(0,T),\\
		&\bdB(\bdx,0)=\bdB_0(x) \quad \text{ in } \quad \Omega,
	\end{aligned}\label{linsys}
\end{equation}
with $\bdF\in L^2(0,T; \bdL^2(\Omega))$ and $\nabla\cdot \bdB_0=0$. 

Since $\bV=H_0(\curl;\Omega)$ is a separable Hilbert space, we can prove the well-posedness of system \eqref{linsys} by Galerkin method. First, for $k=1,2,\cdots$, we select smooth function $\bdw_k=\bdw_k(\bdx)$ such that
$\{\bdw_k\}_{k=1}^\infty$ is an orthonormal basis of $H_0(\curl;\Omega)$
and $\bdL^2(\Omega)$. 
Then, for any positive integer $m$ and $k=1,2,\cdots,m$, we let 
$\bdB_m(\bdx,t)=\sum^m_{k=1}b^k_m(t)\bdw_k(\bdx)$ such that 
\begin{equation}
	(\bdB^\prime_m,\bdw_k)+(\beta(\bdx)\nabla\times\bdB_m,\nabla\times \bdw_k)=(\bdF,\bdw_k),
	\label{galerkin1}
\end{equation}
and
\begin{equation}
	b^k_m(0)=(\bdB_0,\bdw_k).
	\label{galerkin2}
\end{equation}
Or equivalently, we have the following ordinary differential system
\begin{equation}
	\frac{d b^k_m(t)}{dt}+\sum^m_{l=1}a_{kl}b^l_m(t)=f_k(t), \quad k=1,2,\cdots,m\label{ode1}
\end{equation}
with initial conidtion \eqref{galerkin2}. Where $a_{kl}=(\beta(\bdx)\nabla\times\bdw_l,\bdw_k), f_k(t)=(\bdF,\bdw_k)$. According to the standard theory of ordinary differential equations, \eqref{galerkin2}-\eqref{ode1} have a unique solution $(d^1_m(t),d^2_m(t),...,d^m_m(t))$, then we have a unique $\bdB_m(x,t)$ solving \eqref{galerkin1}-\eqref{galerkin2} .
\begin{remark}
	The orthonormal basis $\{\bdw_k\}_{k=1}^\infty$ can be constructed as follows.  As it is known \cite{Monk2003}, for simply connected domain $\Omega$, the kernel space of curl operator $\nabla\times$ in $H_0(\curl;\Omega)$ consists of the gradient of all the functions in $H^1_0(\Omega)$.  That's to say, the double curl operator has an infinite dimensional eigenspace corresponding to eigenvalue 0 \cite{Boffi1999}. Since $H^1_0(\Omega)$ is also a separable Hilbert space, it has an orthonormal basis which  can be chosen as the eigenfunctions of Laplace operator $-\Delta$ in $H^1_0(\Omega)$. Then the orthonormal basis of $H_0(\curl;\Omega)$ can be constructed by collecting the eigenfunctions of $\nabla\times\nabla\times$ in $H_0(\curl;\Omega)$  corresponding to the nonzero eigenvalues and gradient of the orthonormal basis of $H^1_0(\Omega)$. 
\end{remark}
For the solution $\bdB_m$, we have the following estimates.
\begin{lemma}\label{estBm}
	The unique solution $\bdB_m(\bdx,t)$ to problem \eqref{galerkin1}-\eqref{galerkin2} satisfies the following estimate
	\begin{eqnarray}
		\|\bdB_m\|^2_{L^2(0,T; \bV)}\leq C\left(\|\bdB_0\|^2_{\bdL^2(\Omega)}+\|\bdF\|^2_{L^2(0,T;\bdL^2(\Omega))}\right),\label{est1}\\
		\|\bdB_m^\prime\|^2_{L^2(0,T; \bU)}\leq C\left(\|\bdB_0\|^2_{\bdL^2(\Omega)}+\|\bdF\|^2_{L^2(0,T;\bdL^2(\Omega))}\right),\label{est2}
	\end{eqnarray}
	where the constant $C$ depends only on $\beta_0$ and $T$.
\end{lemma}
\begin{proof}
	From \eqref{galerkin1}, we have that 
	\begin{eqnarray*}
		(\bdB_m^\prime,\bdB_m)+(\beta\nabla\times \bdB_m,\nabla\times\bdB_m)=(\bdF,\bdB_m).
	\end{eqnarray*}
	By the assumption on $\beta(\bdx)$, we obtain 
	\begin{eqnarray*}
		\beta_0\|\bdB_m\|^2_{H_0(\curl;\Omega)}\leq (\beta(\bdx)\nabla\times\bdB_m,\nabla\times\bdB_m)+\beta_0\|\bdB_m\|^2_{\bdL^2(\Omega)}
	\end{eqnarray*}
	for $0\leq t\leq T$ and any positive integer $m$. Then we can derive that
	\begin{eqnarray*}
		\frac{1}{2}\frac{d}{dt}\|\bdB_m\|^2_{\bdL^2(\Omega)}+\beta_0\|\bdB_m\|^2_{H(\curl;\Omega)}\leq \beta_0\|\bdB_m\|^2_{\bdL^2(\Omega)}+(\bdF,\bdB_m)
	\end{eqnarray*}
	Since $(\bdF,\bdB_m)\leq 1/2\|\bdF\|^2_{\bdL^2(\Omega)}+1/2\|\bdB_m\|^2_{\bdL^2(\Omega)}$,  we have 
	\begin{eqnarray*}
		\frac{d}{dt}\|\bdB_m\|^2_{\bdL^2(\Omega)}\leq (2\beta_0+1)\|\bdB_m\|^2_{\bdL^2(\Omega)}+\|\bdF\|^2_{\bdL^2(\Omega)}\label{dbBm}
	\end{eqnarray*}
	and
	\begin{eqnarray}
		\|\bdB_m\|^2_{H(\curl;\Omega)}\leq (1+\frac{1}{2\beta_0})\|\bdB_m\|^2_{\bdL^2(\Omega)}+\frac{1}{2\beta_0}\|\bdF\|^2_{\bdL^2(\Omega)}.\label{dbBm0}
	\end{eqnarray}
	Then by the Gronwall's inequality in differential form in Lemma \ref{gronwall}, we have 
	\begin{eqnarray*}
		\|\bdB_m(t)\|^2_{\bdL^2(\Omega)}\leq e^{(2\beta_0+1)t}\left(\|\bdB_m(0)\|^2_{\bdL^2(\Omega)}+\int^t_0\|\bdF(\tau)\|^2_{\bdL^2(\Omega)}d\tau\right).
	\end{eqnarray*}
	By the choice of $\{\bdw_k\}\left(k=1,2,\cdots\right)$, we know that $\|\bdB_m(0)\|^2_{\bdL^2(\Omega)}\leq \|\bdB_0\|^2_{\bdL^2(\Omega)}$, then for any $0\leq t\leq T$, 
	\begin{eqnarray}
		\|\bdB_m(t)\|^2_{\bdL^2(\Omega)} \leq e^{(2\beta_0+1)t}(\|\bdB_0\|^2_{\bdL^2(\Omega)}+\|\bdF\|^2_{L^2(0,T;\bdL^2(\Omega)}).\label{dbBm1}
	\end{eqnarray}
	With the definition 
	\begin{eqnarray*}
		\|\bdB_m\|^2_{L^2(0,T; H(\curl;\Omega))}=\int^T_0\|\bdB_m\|^2_{H(\curl;\Omega)}dt, 
	\end{eqnarray*}
	\eqref{est1} can be derived by using \eqref{dbBm0} and \eqref{dbBm1}.
	
	For any $\bdv \in \bV, \|\bdv\|_{H(\curl;\Omega)}\leq 1$, we have the following decomposition $\bdv=\bdv_1+\bdv_2$, $\bdv_1\in Span\{\bdw_1,\bdw_2,...,\bdw_m\}$ and  $(\bdv_1,\bdv_2)=0$. By the choice of $\{\bdw_k\}\left(k=1,2,\cdots\right)$, we know that $$\|\bdv_1\|_{H(\curl;\Omega)}\leq \|\bdv\|_{H(\curl;\Omega)}\leq 1.$$ 
	Furthermore 
	\begin{eqnarray*}
		(\bdB^\prime_m,\bdv_1)+(\beta(\bdx)\nabla\times\bdB_m,\nabla\times\bdv_1)=(\bdF,\bdv_1).
	\end{eqnarray*}
	Now 
	\begin{eqnarray*}
		\left<\bdB_m^\prime,\bdv\right>=(\bdB_m^\prime,\bdv)=(\bdB_m^\prime,\bdv_1)=(\bdF,\bdv_1)-(\beta(\bdx)\nabla\times\bdB_m,\nabla\times\bdv_1),
	\end{eqnarray*}
	So we have 
	\begin{eqnarray*}
		|\left<\bdB_m^\prime,\bdv\right>|\leq C(\|\bdF\|^2_{\bdL^2(\Omega)}+\|\bdB_m\|^2_{H(\curl;\Omega)}).
	\end{eqnarray*}
	Using the fact that $\|\bdv\|_{H(\curl;\Omega)}\leq 1$,  \eqref{est2} can be proved with the help of \eqref{est1}.
\end{proof}

\begin{theorem}\label{wplinear}
	The linear problem \eqref{linsys} has a unique solution. Furthermore the solution is divergence free $\nabla\cdot \bdB(\bdx,t)=0$ \, a. e. \, $0\leq t\leq T$ provided that $\nabla\cdot\bdB_0(\bdx)=0$.
\end{theorem}
\begin{proof}
	From the Lemma \ref{estBm}, we know that $\{\bdB_m\}^\infty_{m=1}$ is uniformly bounded in $L^2(0,T;\bV)$ and $\{\bdB_m^\prime\}^\infty_{m=1}$ is uniformly bounded in $L^2(0,T;\bU)$. Then there exists a subsequence, still denoted by $\{\bdB_m\}^\infty_{m=1}$  and $\bdB\in L^2(0,T;\bV)$ with $\bdB^\prime\in L^2(0,T;\bU)$ such that
	\begin{eqnarray*}
		\bdB_m\rightharpoonup \bdB \, \mbox{ weakly in } \, \bV,\\
		\bdB^\prime_m\rightharpoonup \bdB^\prime \, \mbox{ weakly in } \, \bU.
	\end{eqnarray*}
	For given positive integer $N$, we chose $\bdv(\bdx,t)=\sum^N_{k=1}v_k(t)\bdw_k(\bdx)$ with smooth $v_k(t)$ for $k=1,2,\cdots,N$ such that $\bdv\in C^1(0,T;\bV)$, then 
	\begin{eqnarray}
		\int^T_0(\bdB_m^\prime,\bdv)+(\beta(\bdx)\nabla\times\bdB_m,\nabla\times\bdv)dt=\int^T_0(\bdF,\bdv)dt.\label{psl0}
	\end{eqnarray}
	By the weak convergence of $\bdB_m$ in $L^2(0,T;\bV)$ and $\bdB_m^\prime$ in $L^2(0,T;\bU)$, we have 
	\begin{eqnarray}
		\int^T_0\left<\bdB^\prime,\bdv\right>+(\beta(\bdx)\nabla\times\bdB,\nabla\times\bdv)dt=\int^T_0(\bdF,\bdv)dt.\label{passinglimit}
	\end{eqnarray}
	Since the functions with the form of $\bdv(\bdx,t)=\sum^N_{k=1}v_k(t)\bdw_k(\bdx)$ are dense in $L^2(0,T;\bV)$, \eqref{passinglimit} is valid for any $\bdv\in L^2(0,T;\bV)$.
	Hence 
	\begin{eqnarray}
		\left<\bdB^\prime,\bdv\right>+(\beta(\bdx)\nabla\times\bdB,\nabla\times\bdv)=(\bdF,\bdv)\label{weak}
	\end{eqnarray}
	for any $\bdv(\bdx)\in H_0(\curl;\Omega)$ and a.e. $0\leq t\leq T$. 
	
	From \eqref{passinglimit}, with integral by part, we have 
	\begin{eqnarray*}
		\int^T_0-\left<\bdB,\bdv^\prime\right>+(\beta(\bdx)\nabla\times\bdB,\nabla\times\bdv)dt=\int^T_0(\bdF,\bdv)dt+(\bdB(\bdx,0),\bdv(\bdx,0))
	\end{eqnarray*}
	for any $\bdv(\bdx,t)\in C^1(0,T;\bV)$ with $\bdv(\bdx,T)=0$. From \eqref{psl0}, we have 
	\begin{eqnarray*}
		\int^T_0-\left<\bdB_m,\bdv^\prime\right>+(\beta(\bdx)\nabla\times\bdB_m,\nabla\times\bdv)dt=\int^T_0(\bdF,\bdv)dt+(\bdB_m(\bdx, 0),\bdv(\bdx, 0)).
	\end{eqnarray*}
	Then by the weak convergence of $\bdB_m$ and $\bdB_m(\bdx,0)\rightarrow\bdB_0(\bdx)$ and letting $m\rightarrow \infty$, we have 
	\begin{eqnarray}
		\int^T_0-\left<\bdB,\bdv^\prime\right>+(\beta(\bdx)\nabla\times\bdB,\nabla\times\bdv)dt=\int^T_0(\bdF,\bdv)dt+(\bdB_0(\bdx),\bdv(0)).\label{bdry}
	\end{eqnarray}
	Since $\bdv(\bdx,0)$ can be chosen arbitrarily, we have that $\bdB(\bdx,0)=\bdB_0(\bdx)$.  This confirms that $\bdB(\bdx,t)$ is the weak solution to \eqref{linsys}. 
	
	Let $\bdB_0(\bdx)=0, \bdF(\bdx,t)=0$ and chose $\bdv=\bdB$ in \eqref{weak}, we have 
	\begin{eqnarray*}
		\frac{1}{2}\frac{d}{dt}\|\bdB\|^2_{\bdL^2(\Omega)}+(\beta(\bdx)\nabla\times\bdB,\nabla\times\bdB)=0
	\end{eqnarray*}
	then $\|\bdB(t)\|^2_{\bdL^2(\Omega)}$ is a constant and 
	\begin{eqnarray*}
		\|\bdB(\bdx, t)\|^2_{\bdL^2(\Omega)}=\|\bdB_0(x)\|^2_{\bdL^2(\Omega)}=0, \quad \forall \, t\in [0,T].
	\end{eqnarray*}
	This confirms that the solution is unique.
	
	It is known that $H^1_0(\Omega)=\{\phi\in L^2(\Omega), \nabla\phi\in\bdL^2(\Omega), \phi=0 \mbox{ on } \partial\Omega\}$ is a separable Hilbert space, we can chose $w_k(\bdx),k=1,2,\cdots$ such that 
	\begin{eqnarray*}
		\{w_k(\bdx)\}^\infty_{k=1} \mbox{ is an orthonormal basis of }  H^1_0(\Omega).
	\end{eqnarray*}
	We chose smooth $\psi_k(t)$ such that $\psi_k(T)=0,k=1,2,\cdots,N$,  let $\bdv(\bdx,t)=\sum^N_{k=1}\psi_k(t) \nabla w_k(\bdx)$ in \eqref{psl0}. 
	Because $w_k(\bdx)\in H^1_0(\Omega)$,  $\bdv(\bdx,t)\in H_0(\curl;\Omega)$ for any $t\in [0,T]$. We have from \eqref{bdry} that
	\begin{eqnarray}
		\int^T_0(\bdB,\bdv^\prime)=0 \label{divB}
	\end{eqnarray}
	because of $\nabla\cdot\bdF=0, \nabla\cdot\bdB_0(\bdx)=0$. Since $\{w_k(\bdx)\}_{k=1}^\infty$ is a basis of $H^1_0(\Omega)$,  we have from \eqref{divB} that
	$$(\bdB(\bdx,t), \nabla\phi)=0 \, \mbox{ a.e. } \, 0\leq t\leq T$$
	for any $\phi(x)\in H^1_0(\Omega)$.
	This confirms that $\nabla\cdot\bdB(\bdx,t)$ is divergence free a.e. $0\leq t\leq T$.
\end{proof}

\begin{remark}
	We remark that the results in this subsection is still valid if we assume $\bdF(\bdx,t)\in L^2(0,T; \bU)$. 
\end{remark}

\subsection{The well-posedness of nonlinear system}
In the section, we will prove the well-posedness of \eqref{BF1} with the help of Banach fixed point Theorem \ref{fpt}. Let the operator $\bG: L^2(0,T;\bdL^2(\Omega))\rightarrow L^2(0,T; \bdL^2(\Omega))$ 
\begin{eqnarray*}
	\bG(\bdB)=R_\alpha(\frac{f}{1+\sigma|\bdB|^2}\bdB)+R_m(\bdu\times\bdB).
\end{eqnarray*}
Where $f$ is supported in $\Omega_3$ and $\bdu$ is supported in $\Omega_2$. We further assume that $f\in H^1(0,T;L^\infty(\Omega))$ and $\bdu\in H^1(0,T; \bdL^\infty(\Omega))$.
We have the following result on the bounded property of operator $\bG$.
\begin{lemma}\label{estG}
	For any $\bdB_1, \bdB_2 \in L^2(0,T; \bdL^2(\Omega))$, there is a constant $C$ which depends on $R_\alpha, R_m, \sigma$ and $f \in L^\infty(\Omega)$ supported in $\Omega_3$, $\bdu \in \boldsymbol{L}^\infty(\Omega)$ supported in $\Omega_2$, such that 
	\begin{eqnarray*}
		\|\bG(\bdB_1)-\bG(\bdB_2)\|_{\bdL^2(\Omega)}\leq C\|\bdB_1-\bdB_2\|_{\bdL^2(\Omega)}\label{boundF}.
	\end{eqnarray*}
\end{lemma}
\begin{proof}
	By the definition, we have 
	\begin{eqnarray*}
		&&\|\bG(\bdB_1)-\bG(\bdB_2)\|^2_{\bdL^2(\Omega)}=R_\alpha\int_{\Omega_3}f^2|\frac{\bdB_1}{1+\sigma|\bdB_1|^2}-\frac{\bdB_2}{1+\sigma|\bdB_2|^2}|^2dx\\
		&&\qquad\quad+R_m\int_{\Omega_2}|\bdu\times(\bdB_1-\bdB_2)|^2dx\\
		&&\qquad=R_\alpha\int_{\Omega_3}f^2|\frac{(\bdB_1-\bdB_2)+\sigma|\bdB_2|^2(\bdB_1-\bdB_2)+\sigma(|\bdB_2|^2-|\bdB_1|^2)\bdB_2}{(1+\sigma|\bdB_1|^2)(1+\sigma|\bdB_2|^2)}|^2dx\\
		&&\qquad\quad+R_m\int_{\Omega_2}|\bdu\times(\bdB_1-\bdB_2)|^2dx\\
		&&\qquad\leq C\|\bdB_1-\bdB_2\|^2_{\bdL^2(\Omega)}.
	\end{eqnarray*}
\end{proof}
Now we are in position to prove the existence and uniqueness of solution to the nonlinear system \eqref{BF1}.  
\begin{theorem}
	\label{t34}
	There exists a unique solution of problem~\eqref{BF1}.
\end{theorem}
\begin{proof}
	For any $\bdB\in L^2(0,T;\bV)$, we define $\bdF: L^2(0,T; \bdL^2(\Omega))\rightarrow L^2(0,T; \bU)$ by $\bdF(\bdB)=\nabla\times\bG(\bdB)$. We consider the following problem: find $\bdA\in L^2(0,T;\bV)$ with $\bdA^\prime\in L^2(0,T;\bU)$ such that
	\begin{eqnarray}
		(\bdA^\prime, \bdv)+(\beta(\bdx)\nabla\times\bdA,\nabla\times \bdv)=\left<\bdF(\bdB),\bdv\right> \, \mbox{ a.e. } \, 0\leq t\leq T
		\label{t33}
	\end{eqnarray}
	for any $\bdv\in H_0(\curl;\Omega)$.  Theorem \ref{wplinear} tells us that the above problem has a unique solution $\bdA$ in $L^2(0,T;\bV)$. Then we can define $S: L^2(0,T;\bdL^2(\Omega))\rightarrow L^2(0,T;\bdL^2(\Omega))$ by letting $S(\bdB)=\bdA$, which satisfies \eqref{t33} and $\bdn\times\bdA(\bdx,t)|_{\partial\Omega}=0$.
	
	In the following, we are going to prove that $S$ is a contracting operator if $T$ is small enough. Given $\bdB, \tilde{\bdB} \in L^2(0,T;\bdL^2(\Omega))$, let $\bdA=S(\bdB), \tilde{\bdA}=S( \tilde{\bdB})$, it is easy to derive that 
	\begin{eqnarray*}
		&&\frac{d}{dt}\|\bdA-\tilde{\bdA}\|^2_{\bdL^2(\Omega)}+2(\beta(\bdx)\nabla\times(\bdA-\tilde{\bdA}), \nabla\times(\bdA-\tilde{\bdA}))\\
		&&\qquad =2\left<\bdA-\tilde{\bdA},\bdF(\bdB)-\bdF(\tilde{\bdB})\right>\\
		&&\qquad = 2(\nabla\times(\bdA-\tilde\bdA), \bG(\bdB)-\bG(\tilde\bdB))\\
		&&\qquad\leq 2\varepsilon \|\nabla\times(\bdA-\tilde{\bdA})\|^2_{\bdL^2(\Omega)}+\frac{2}{\varepsilon}\|\bG(\bdB)-\bG(\tilde{\bdB})\|^2_{\bdL^2(\Omega)}.
	\end{eqnarray*}
	Then for sufficiently small $\varepsilon$ (we can chose $\varepsilon\leq \beta_0$ here), with the help of Lemma \ref{estG},  we have
	\begin{eqnarray*}
		\frac{d}{dt}\|\bdA-\tilde{\bdA}\|^2_{\bdL^2(\Omega)}\leq C\|\bG(\bdB)-\bG(\tilde\bdB)\|^2_{\bdL^2(\Omega)}\leq C\|\bdB-\tilde\bdB\|^2_{\bdL^2(\Omega)} .
	\end{eqnarray*}
	So 
	\begin{eqnarray*}
		\|\bdA(t)-\tilde{\bdA}(t)\|^2_{\bdL^2(\Omega)}\leq C\int^t_0\|\bdB(\tau)-\tilde\bdB(\tau)\|^2_{\bdL^2(\Omega)}d\tau\leq C\|\bdB-\tilde\bdB\|^2_{L^2(0,T;\bdL^2(\Omega))}
	\end{eqnarray*}
	and 
	\begin{eqnarray*}
		\|\bdA-\tilde{\bdA}\|^2_{L^2(0,T;\bdL^2(\Omega))}\leq  CT\|\bdB-\tilde\bdB\|^2_{L^2(0,T;\bdL^2(\Omega))}.
	\end{eqnarray*}
	It is just
	\begin{eqnarray*}
		\|S(\bdB)-S(\tilde{\bdB})\|_{L^2(0,T;\bdL^2(\Omega))}\leq  (CT)^{1/2}\|\bdB-\tilde\bdB\|_{L^2(0,T;\bdL^2(\Omega))}.
	\end{eqnarray*}
	We can conclude that $S$ is a contracting operator for sufficiently small $T$. For any given $T>0$, there exists  $T_1\leq T$ such that $CT_1<1$. Then with the help of Banach fixed point theorem, there is a solution of problem \eqref{BF1} in $[0,T_1]$. By the definition of $S(\bdB)$, we know that the solution $\bdB$ is actually in  $L^2(0,T_1; H_0(\curl;\Omega))$ and 
	$\nabla\cdot\bdB(\bdx,t)=0$ a.e. $t\in [0, T_1]$.
	Then we can recursively use the above procedure to prove that \eqref{BF1} has a solution in $[T_1, 2T_1], [2T_1, 3T_1]$ and so on. Further, in finite steps, we prove that
	there exists a solution of \eqref{BF1} in $[0,T]$.
	
	Finally, we shall prove the uniqueness of this solution.
	Take two solutions $\bdB_1$ and $\bdB_2 \in  L^{2}(0, T; \bdv)$ that satisfy the system \eqref{BF1}.
	And set $\bar{\bdB}=\bdB_1-\bdB_2$, which satisfies  $\bar{\bdB}(0)=0$. 
	Then, substituting $\bdB_1$ and $\bdB_2$ into system \eqref{BF1} and subtracting them yields:
	\begin{equation*}
		\begin{aligned}
			\left<\bar{\bdB}^{\prime} , \bdA\right>
			+  \left(\beta \nabla \times \bar{\bdB} , \nabla \times \bdA\right)
			= R_\alpha \left(\frac{f\bdB_1}{1+\sigma\left|\bdB_1\right|^2}
			-\frac{f\bdB_2}{1+\sigma\left|\bdB_2\right|^2} , \nabla \times \bdA\right)
			+ R_m
			\left(\bdu \times \bar{\bdB}, \nabla \times \bdA\right).
		\end{aligned}
	\end{equation*}
	Taking $\bdA=\bar{\bdB}$ and with the help of Lemma \ref{estG}, we have
	\begin{equation*}
		\begin{aligned}
			&\frac{1}{2} \frac{d}{d t}\left\|\bar{\bdB}\right\|^2+\left\|\nabla \times \bar{\bdB}\right\|_\beta^2 
			\leq  C \left\|\frac{ f \bdB_1}{1+\sigma\left|\bdB_1\right|^2}-\frac{ f \bdB_2}{1+\sigma\left|\bdB_2\right|^2}\right\| \left\|\nabla \times \bar{\bdB}\right\|_\beta 
			+ C \left\|\bdu \times \bar{\bdB}\right\| \left\|\nabla \times \bar{\bdB}\right\|_\beta \\
			&\quad\quad \leq  \left\|\nabla \times \bar{\bdB}\right\|_\beta^2 
			+ C \left\|\bar\bdB\right\|_{L^2(\Omega)}^2
			+ C \left\|\bdu \times \bar{\bdB}\right\|_{L^2(\Omega)}^2.
		\end{aligned}
	\end{equation*}
	Where $\left\|\nabla \times \bar{\bdB}\right\|_\beta=(\beta\nabla\times \bdB,\nabla\times\bdB)$ and the inequality holds since $0\leq \beta_0\leq\beta\leq \beta_M $. Because $\bdu\in L^\infty(\Omega)$, we can conclude that 
	\begin{equation*}
		\frac{d}{d t}\left\|\bar{\bdB}\right\|^2  \leq  C \left\|\bar\bdB\right\|_{L^2(\Omega)}^2.
	\end{equation*}
	
	Which implies $\bar{\bdB}=0$ by applying the Gronwall's inequality in Lemma \ref{gronwall}.  
	Then the solution is unique and we finish the proof of Theorem~\ref{t34}. 
\end{proof}



\section{The numerical scheme}\label{sect4}
In the first two sections, the new spherical interface dynamo model~\eqref{Z1} and the relevant properties have been introduced in detail, and then we shall seek the numerical solution of the model.
In this section, we mainly introduce the finite element approximation of the variational formulation \eqref{BF1}.
For the temporal discretization, the backward Euler scheme is used in the left hand side and the right hand side is explicit in time. For spatial discretization, we use the edge element method.
And the corresponding fully discretized numerical scheme is provided in the subsections below. Meanwhile, we investigate the properties of this scheme such as the stability and the well-posedness.


\subsection{The full discrete scheme}
In this subsection, we introduce the edge element method for solving the system \eqref{BF1} and provide the corresponding fully discretized scheme.

We first make the uniform partition of time interval $\left[0,T\right]$ as follows
$$
0=t_0<t_1<t_2<\cdots<t_K=T,
$$
where $t_n = n \tau \; (n=0,1,\cdots,K)$ and time step $\tau = \frac{T}{K}$.
For any time discrete sequence $\left\{u^n\right\}^K_{n=0}$, we define $u^n(\cdot) = u(\cdot, t_n)$ and let
\begin{equation*}
	\partial_\tau u^n = \frac{u^n - u^{n-1}}{\tau}, \quad n \in \left[1,K\right].
\end{equation*}

For spatial discretization, we first define the triangulation of the spherical domain $\Omega$. We assume that its outer boundary $\partial \Omega$ is a closed convex polygon which approximates the boundary of the real spherical surface.  
Suppose $\mathcal{M}_h$ is a regular tetrahedral triangulation of domain $\Omega$ and is divided into four single triangulations $\mathcal{M}_h^i(i=1,2,3,4)$ whose boundary vertices match at the interfaces of nearby regions. Then the mesh $\mathcal{M}_h$ is an approximated partition of real spherical domain.
On this basis, we use the edge finite element to approximate the system \eqref{BF1}.
The Nédélec edge element space $\boldsymbol{S}_{h}$ defined by \cite{Nedelec1980,Nedelec1986} over $\mathcal{M}_h$  is
$$\boldsymbol{S}_{h} \triangleq 
\left\{\bdA \in \boldsymbol{V}:\;\bdA|_{\mathcal{T}}=\boldsymbol{a}_{\mathcal{T}}+\boldsymbol{b}_{\mathcal{T}} \times \bdx \;\; \text{with} \;\; \boldsymbol{a}_{\mathcal{T}},\boldsymbol{b}_{\mathcal{T}} \in \mathbb{R}^3 \;\; \forall \; \mathcal{T} \in \mathcal{M}_h\right\}.$$
Further, we propose the following fully discrete finite element approximation of the variational problem \eqref{BF1}. 
For any $\bdA_h \in \boldsymbol{S}_{h}$, find $\bdB^n_h \in \boldsymbol{S}_{h} \; \left(n = 1,2,3,\cdots,K\right)$, such that
\begin{equation}
	\begin{aligned}
		\left(\partial_\tau \bdB^n_h, \bdA_h\right)+&\left(\beta_h \nabla \times \bdB^n_h, \nabla \times \bdA_h\right) \\
		=& R_{\alpha}\left(\frac{f^n_h}{1+\sigma|\bdB^{n-1}_h|^{2}} \bdB^{n-1}_h, \nabla \times \bdA_h\right)+R_{m}\left(\bdu^n_h \times \bdB^{n-1}_h, \nabla \times \bdA_h\right),
	\end{aligned}
	\label{EFE1}
\end{equation}
with initial condition $\bdB^0_h = \bdB_0$. Since $\nabla \cdot \bdB_{0} = 0$,  $\left(\bdB_{0},\nabla \phi_h\right)=0$ for any $\phi_h \in V_h \cap H_0^1$, where $V_h$ is the linear Lagrange finite element space over the mesh $\mathcal{M}_h$. Further, this initial condition makes the above system satisfy that $\left(\bdB^{n}_{h},\nabla \phi_h\right)=0$ for any $n$.
\subsection{Well-posedness and stability analysis}
The following lemma gives the well-posedness and the stability of the fully discrete numerical scheme \eqref{EFE1}.
\begin{theorem}
	There exists a unique solution $\bdB^n_h$ to the approximate variational problem \eqref{EFE1} for each fixed $t_n \; \left(n = 1,2,3,\cdots,K\right)$.
\end{theorem}
\begin{proof} 
	We use the similar technique as \cite{Chan2006} to prove the existence and uniqueness of the solution to the system \eqref{EFE1}. We first define
	a mapping $F_h : \bar{\bdB}_h \rightarrow \bdB_h$ by
	\begin{equation}
		a_h\left(\bdB_h, \bdA_h\right)
		= b_h\left(\bar{\bdB}_h, \bdA_h\right), \quad \forall \; \bdA_h \in \boldsymbol{S}_{h},
		\label{EFE2}
	\end{equation}
	where
	\begin{equation*}
		\begin{aligned}
			a_h\left(\bdB, \bdA\right)&=\left(\bdB, \bdA\right)+\tau\left(\beta_h \nabla \times \bdB, \nabla \times \bdA\right), \\
			b_h\left(\bdB, \bdA\right)&= \left(\bdB^{n-1}_h, \bdA\right)
			+\tau R_{\alpha}\left(\frac{f^n_h}{1+\sigma|\bdB^{n-1}_h|^{2}} \bdB, \nabla \times \bdA\right)
			+\tau R_{m}\left(\bdu^n_h \times \bdB, \nabla \times \bdA\right).
		\end{aligned}
	\end{equation*}
	It is clear that $a_h\left(\bdB, \bdA\right)$ is coercive in $\boldsymbol{S}_h$, thus there is a unique solution $\bdB_h \in \boldsymbol{S}_{h}$ for any fixed $\bar{\bdB}_h$.
	Then the mapping $F_h$ is well-defined.
	
	We next take $\bdA_h = \bdB_h$ in \eqref{EFE2} and use Young's inequality to obtain
	\begin{equation*}
		\begin{aligned}
			&\left\|\bdB_h\right\|^2 + \tau \left\|\nabla \times \bdB_h\right\|^2_\beta \\
			\leq & \left\|\bdB^{n-1}_h\right\| \left\|\bdB_h\right\| 
			+ \tau R_\alpha \left|f_{max}\right| \left\|\bar{\bdB}_h\right\| \left\|\nabla \times \bdB_h\right\| 
			+ \tau R_m \left|\bdu_{max}\right| \left\|\bar{\bdB}_h\right\| \left\|\nabla \times \bdB_h\right\| \\
			\leq & \frac{1}{2} \left\|\bdB^{n-1}_h\right\|^2 
			+ \frac{1}{2} \left\|\bdB_h\right\|^2
			+ \frac{\tau}{2}\left\|\nabla \times \bdB_h\right\|^2_\beta
			+ \frac{\tau R_\alpha^2}{\beta_h} \left|f_{max}\right|^2 		 
			\left\|\bar{\bdB}_h\right\|^2   	
			+ \frac{\tau R_m^2}{\beta_h} \left|\bdu_{max}\right|^2 
			\left\|\bar{\bdB}_h\right\|^2 \\
			\leq & \left\|\bdB^{n-1}_h\right\|^2 
			+ \frac{2 \tau}{\beta_h} \left(R_\alpha^2 \left|f_{max}\right|^2 + R_m^2 \left|\bdu_{max}\right|^2\right) \left\|\bar{\bdB}_h\right\|^2 .  
		\end{aligned}
	\end{equation*}
	Therefore, for any $\bar{\bdB}_h$ located in  $\boldsymbol{K}_h = \left\{\bdA_h | \bdA_h\in\boldsymbol{S}_{h}, \left\|\bdA_h\right\| \leq r \right\}$ with $r^2 = 2 \left\|\bdB^{n-1}_h\right\|^2$, then the following inequality holds
	\begin{equation*}
		\left\|\bdB_h\right\|^2 + \tau \left\|\nabla \times \bdB_h\right\|^2_\beta \leq r^2,
	\end{equation*}
	provided that $\tau$ is sufficiently small to make $\frac{4 \tau}{\beta_h} \left(R_\alpha^2 \left|f_{max}\right|^2 + R_m^2 \left|\bdu_{max}\right|^2\right) \leq 1$.
	The above result shows that the mapping $F_h$ is a continuous mapping from the bounded set $\boldsymbol{K}_h$ into itself.
	Further utilizing the Brouwer’s fixed point theorem, we can obtain the existence of the solution to system \eqref{EFE1}.  
	
	Then, we shall prove the uniqueness of this solution by contradiction.
	We take two different solutions $\bdB^n_{h,1}$ and $\bdB^n_{h,2} \in  \boldsymbol{S}_{h}$ that satisfy the system \eqref{EFE1}.
	And set $\bar{\bdB}^n_h=\bdB^n_{h,1}-\bdB^n_{h,2}$ which satisfies $\bar{\bdB}^0_h=0$. 
	Then, substituting $\bdB^n_{h,1}$ and $\bdB^n_{h,2}$ into system \eqref{EFE1} and subtracting them yields:
	\begin{equation*}
		\begin{aligned}
			\left(\partial_\tau  \bar{\bdB}^n_{h}, \bdA_{h}\right)
			+ & \left(\beta_h \nabla \times \bar{\bdB}^n_{h} , \nabla \times \bdA_{h}\right)\\
			= &R_\alpha \left(\frac{f^n_h\bdB^{n-1}_{h,1}}{1+\sigma\left|\bdB^{n-1}_{h,1}\right|^2}
			-\frac{f^n_h\bdB^{n-1}_{h,2}}{1+\sigma\left|\bdB^{n-1}_{h,2}\right|^2} , \nabla \times \bdA_{h}\right)
			+ R_m
			\left(\bdu^n_h \times \bar{\bdB}^{n-1}_{h}, \nabla \times \bdA_{h}\right).
		\end{aligned}
	\end{equation*}
	Taking $\bdA_{h}=\tau \bar{\bdB}^n_{h}$ yields
	\begin{equation*}
		\begin{aligned}
			&\frac{1}{2}\left\|\bar{\bdB}_h^n\right\|^2-\frac{1}{2}\left\|\bar{\bdB}_h^{n-1}\right\|^2+\beta_h \tau\left\|\nabla \times \bar{\bdB}_h^n\right\|^2 \\
			\leq & C \left\|\frac{ f^n_h \bdB^{n-1}_{h,1}}{1+\sigma\left|\bdB^{n-1}_{h,1}\right|^2}-\frac{ f^n_h \bdB^{n-1}_{h,2}}{1+\sigma\left|\bdB^{n-1}_{h,2}\right|^2}\right\| \left\|\nabla \times \bar{\bdB}^{n}_{h}\right\|_{\beta_h} 
			+ C \left\|\bdu^n_h \times \bar{\bdB}^{n-1}_{h}\right\| \left\|\nabla \times \bar{\bdB}^{n}_{h}\right\|_{\beta_h} \\
			\leq & \frac{1}{2} \beta_h \tau  \left\|\nabla \times \bar{\bdB}^{n}_{h}\right\|^2 
			+ C \left\|\frac{ f^n_h \bdB^{n-1}_{h,1}}{1+\sigma\left|\bdB^{n-1}_{h,1}\right|^2}-\frac{ f^n_h \bdB^{n-1}_{h,2}}{1+\sigma\left|\bdB^{n-1}_{h,2}\right|^2}\right\|^2
			+ C \left\|\bdu^n_h \times \bar{\bdB}^{n-1}_{h}\right\|^2.
		\end{aligned}
	\end{equation*}
	Then, accumulating it concerning n from 1 to K, and using the facts that $f\in L^\infty(\Omega)$ and $\bdu\in\bdL^\infty(\Omega)$, we obtain
	\begin{equation*}
		\begin{aligned}
			&\left\|\bar{\bdB}_h^K\right\|^2 + \tau \sum_{n=1}^K \left\|\nabla \times \bar{\bdB}_h^n\right\|_{\beta_h}^2  \\
			\leq & \left\|\bar{\bdB}_h^0\right\|^2
			+ C \sum_{n=1}^K \left\|\frac{ f^n_h \bdB^{n-1}_{h,1}}{1+\sigma\left|\bdB^{n-1}_{h,1}\right|^2}
			-\frac{ f^n_h \bdB^{n-1}_{h,2}}{1+\sigma\left|\bdB^{n-1}_{h,2}\right|^2}\right\|^2 
			+ C \sum_{n=1}^K \left\|\bdu^n_h \times \bar{\bdB}^{n-1}_{h}\right\|^2 \\
			\leq & C \sum_{n=1}^K \left\|\frac{f^n_h\left|\bar{\bdB}^{n-1}_{h}\right|
				+\sigma f^n_h\left|\bar{\bdB}^{n-1}_{h}\right| \left|\bdB^{n-1}_{h,2}\right|^2
				+\sigma f^n_h\left|\bar{\bdB}^{n-1}_{h}\right| \left|\bdB^{n-1}_{h,2}\right|\left|\bdB^{n-1}_{h,1}+\bdB^{n-1}_{h,2}\right|}{\left(1+\sigma\left|\bdB^{n-1}_{h,1}\right|^2\right)\left(1+\sigma\left|\bdB^{n-1}_{h,2}\right|^2\right)}\right\|^2 \\ 
			&+ C \sum_{n=1}^K \left\|\bdu^n_h \times \bar{\bdB}^{n-1}_{h}\right\|^2
			\leq C \sum_{n=0}^{K-1} \left\|\bar{\bdB}^{n}_{h}\right\|^2,
		\end{aligned}
	\end{equation*}
	which implies $\bar{\bdB}^{n}_{h}=0$ by applying the discrete Gronwall inequality.
\end{proof}
We next derive stability estimates of the solution to \eqref{EFE1}.
\begin{theorem}
	For the solution $\left\{\bdB^n_h\right\}^K_{n=0}$ to \eqref{EFE1}, 
	there exists a positive constant $C$ independent of $h$,  such that the following stability estimates hold.
	\begin{equation}
		\max_{1 \leq n \leq K}\left\|\bdB_h^n\right\|^2+\tau \sum_{n=1}^K \left\|\nabla \times \bdB_h^n\right\|_{\beta_h}^2 \leq C\left\|\bdB^0_h\right\|_{\boldsymbol{H}}^2,
		\label{WD3}
	\end{equation}
	\begin{equation}
		\max_{1 \leq n \leq K}\left\|\nabla \times \bdB_h^n\right\|_{\beta_h}^2 +\tau \sum_{n=1}^K \left\|\partial_\tau \bdB_h^n\right\|^2  \leq C\left\|\bdB^0_h\right\|_{\boldsymbol{H}}^2.
		\label{WD4}
	\end{equation}
\end{theorem}
\begin{proof}
	For the first inequality, we first take $\bdA_h=\tau \bdB_h^n$ in \eqref{EFE1} and use the lemmas from the preliminaries, which yields
	\begin{equation*}
		\begin{aligned}
			&\frac{1}{2}\left\|\bdB_h^n\right\|^2-\frac{1}{2}\left\|\bdB_h^{n-1}\right\|^2+\beta_h \tau\left\|\nabla \times \bdB_h^n\right\|^2 \\
			\leq & R_\alpha \tau\left\|f^n_h \bdB_h^{n-1}\right\|\left\|\nabla \times \bdB_h^n\right\| 
			+ R_m \tau\left\|\bdu^n_h \times \bdB_h^{n-1}\right\| \left\|\nabla \times \bdB_h^n\right\| \\
			\leq & C \tau\left\|\bdB_h^{n-1}\right\|^2 + \frac{1}{2} \beta_h \tau\left\|\nabla \times \bdB_h^n\right\|^2.
		\end{aligned}
	\end{equation*}
	Simplifying the above equation, we can obtain 
	\begin{equation*}
		\left\|\bdB_h^n\right\|^2-\left\|\bdB_h^{n-1}\right\|^2+\beta_h \tau\left\|\nabla \times \bdB_h^n\right\|^2
		\leq  C \tau\left\|\bdB_h^{n-1}\right\|^2.
	\end{equation*}
	Then, accumulating it concerning $n$ from 1 to $K$, which yields
	\begin{equation*}
		\left\|\bdB_h^K\right\|^2 + \tau \sum_{n=1}^K \left\|\nabla \times \bdB_h^n\right\|_{\beta_h}^2 
		\leq \left\|\bdB_h^0\right\|^2 
		+ C \tau \sum_{n=0}^{K-1} \left\|\bdB_h^n\right\|^2.
	\end{equation*}
	Finally, using the discrete Gronwall's inequality, we can obtain
	\begin{equation*}
		\left\|\bdB_h^n\right\|^2+\tau \sum_{n=1}^K \left\|\nabla \times \bdB_h^n\right\|_{\beta_h}^2 
		\leq C \left\|\bdB^0_h\right\|^2
		\leq C \left(\left\|\bdB^0_h\right\|^2 + \left\|\nabla \times \bdB^0_h\right\|^2\right).
	\end{equation*}
	which implies \eqref{WD3}.
	
	For the second stability estimates \eqref{WD4}, taking
	$ \bdA_{h} = \tau \partial_\tau \bdB_{h}^{n} =  \bdB_{h}^{n} - \bdB_{h}^{n-1} $, we obtain
	\begin{equation*}
		\begin{aligned}
			&\tau\left\|\partial_{\tau} \bdB_{h}^{n}\right\|^{2}
			+  \left(\beta_h \nabla \times \bdB_{h}^{n}, \nabla \times \bdB_{h}^{n}\right) 
			= \tau\left\|\partial_{\tau} \bdB_{h}^{n}\right\|^{2}
			+  \left\|\nabla \times \bdB_{h}^{n}\right\|_{\beta_h}^2\\
			=&\left(\beta_h \nabla \times \bdB_{h}^{n}, \nabla \times \bdB_{h}^{n-1}\right)
			+  R_{\alpha}  \left(\frac{f^n_h \bdB_{h}^{n-1}}{1+\sigma\left|\bdB_{h}^{n-1}\right|^{2}}  , \tau\nabla  \times \partial_{\tau}\bdB_{h}^{n}\right)   +  R_{m} \left(\bdu^n_h \times \bdB_{h}^{n-1} , \tau\nabla  \times \partial_{\tau}\bdB_{h}^{n}\right) \\
			\leq & \frac{1}{2}\left\|\nabla \times \bdB_{h}^{n}\right\|_{\beta_h}^2 
			+ \frac{1}{2}\left\|\nabla \times \bdB_{h}^{n-1}\right\|_{\beta_h}^2\\
			&+  R_{\alpha}  \left(\frac{f^n_h \bdB_{h}^{n-1}}{1+\sigma\left|\bdB_{h}^{n-1}\right|^{2}}  , \tau\nabla  \times \partial_{\tau}\bdB_{h}^{n}\right)   +  R_{m} \left(\bdu^n_h \times \bdB_{h}^{n-1} , \tau\nabla  \times \partial_{\tau}\bdB_{h}^{n}\right).
		\end{aligned}
	\end{equation*}
	Simplifying the above inequality 
	\begin{equation*}
		\begin{aligned}
			&\tau\left\|\partial_{\tau} \bdB_{h}^{n}\right\|^{2}
			+  \frac{1}{2} \left(\left\|\nabla \times \bdB_{h}^{n}\right\|_{\beta_h}^2 - \left\|\nabla \times \bdB_{h}^{n-1}\right\|_{\beta_h}^2\right) \\
			\leq & R_{\alpha}  \left(\frac{f^n_h \bdB_{h}^{n-1}}{1+\sigma\left|\bdB_{h}^{n-1}\right|^{2}}  , \tau\nabla  \times \partial_{\tau}\bdB_{h}^{n}\right)   +  R_{m} \left(\bdu^n_h \times \bdB_{h}^{n-1} ,\tau \nabla  \times \partial_{\tau}\bdB_{h}^{n}\right),
		\end{aligned}
	\end{equation*}
	and summing up the above inequality from $1$ to $K$ yields
	\begin{equation}
		\begin{aligned}
			&\tau \sum_{n=1}^{K} \left\|\partial_{\tau} \bdB_{h}^{n}\right\|^{2}
			+  \frac{1}{2} \left\|\nabla \times \bdB_{h}^{K}\right\|_{\beta_h}^2 \\
			\leq & \frac{1}{2} \left\|\nabla \times \bdB_{h}^{0}\right\|_{\beta_h}^2
			+ R_{\alpha} \sum_{n=1}^{K} \left(\frac{f^n_h \bdB_{h}^{n-1}}{1+\sigma\left|\bdB_{h}^{n-1}\right|^{2}}  , \tau\nabla  \times \partial_{\tau}\bdB_{h}^{n}\right)   
			+  R_{m} \sum_{n=1}^{K} \left(\bdu^n_h \times \bdB_{h}^{n-1} , \tau\nabla  \times \partial_{\tau}\bdB_{h}^{n}\right).
		\end{aligned}
		\label{WD5}
	\end{equation}
	For the second term at the right hand side of \eqref{WD5}, after rearranging the summation, we have  
	\begin{equation*}
		\begin{aligned}
			&\sum_{n=1}^{K} \int_{\Omega} \frac{f^{n}_h \bdB_{h}^{n-1}}{1+\sigma\left|\bdB_{h}^{n-1}\right|^{2}}  \cdot \nabla \times\left(\bdB_{h}^{n}-\bdB_{h}^{n-1}\right) d \bdx\\
			= & \int_{\Omega}\frac{f^{K}_h \bdB_{h}^{K-1}}{1+\sigma\left|\bdB_{h}^{K-1}\right|^{2}} \cdot \nabla \times \bdB_{h}^{K} d \bdx
			- \int_{\Omega} \frac{f^{1}_h \bdB_{h}^{0}}{1+\sigma\left|\bdB_{h}^{0}\right|^{2}} \cdot \nabla \times \bdB_{h}^{0} d \bdx \\
			& - \sum_{n=1}^{K-1} \int_{\Omega} \left(\frac{f^{n+1}_h \bdB_{h}^{n}}{1+\sigma\left|\bdB_{h}^{n}\right|^{2}}
			-\frac{f^{n}_h \bdB_{h}^{n-1}}{1+\sigma\left|\bdB_{h}^{n-1}\right|^{2}}\right) \cdot \nabla \times \bdB_{h}^{n}d \bdx \\
			=& \int_{\Omega}\frac{f^{K}_h \bdB_{h}^{K-1}}{1+\sigma\left|\bdB_{h}^{K-1}\right|^{2}} \nabla \times \bdB_{h}^{K}  d \bdx
			- \int_{\Omega} \frac{f^{1}_h \bdB_{h}^{0}}{1+\sigma\left|\bdB_{h}^{0}\right|^{2}} \nabla \times \bdB_{h}^{0} d \bdx \\
			& - \sum_{n=1}^{K-1} \int_{\Omega}  \frac{f^{n+1}_h \bdB_{h}^{n} - f^{n}_h \bdB_{h}^{n-1} + f^{n+1}_h \bdB_{h}^{n} \sigma \left| \bdB_{h}^{n-1} \right|^{2} - f^{n}_h \bdB_{h}^{n-1} \sigma\left|\bdB_{h}^{n}\right|^{2} }{\left(1+\sigma\left|\bdB_{h}^{n}\right|^{2}\right)\left(1+\sigma\left|\bdB_{h}^{n-1}\right|^{2}\right)} \cdot\nabla \times \bdB_{h}^{n} d \bdx. 
		\end{aligned}
	\end{equation*}
	Then, applying the Young's inequality, we can obtain
	\begin{equation*}
		\begin{aligned}
			& \left|\sum_{n=1}^{K} \int_{\Omega} \frac{f^{n}_h \bdB_{h}^{n-1}}{1+\sigma\left|\bdB_{h}^{n-1}\right|^{2}}  \cdot \nabla \times\left(\bdB_{h}^{n}-\bdB_{h}^{n-1}\right) d \bdx\right|\\
			\leq & \frac{1}{8} \left\|\nabla \times \bdB_{h}^{K} \right\|_{\beta_h}^2
			+ C \left\|\bdB_{h}^{K-1}\right\|^2
			+ \frac{1}{8} \left\|\nabla \times \bdB_{h}^{0} \right\|_{\beta_h}^2
			+ C \left\|\bdB_{h}^{0}\right\|^2 \\
			& + \sum_{n=1}^{K-1} \int_{\Omega}  \frac{\left|f^{n+1}_h \left(\bdB_{h}^{n} -  \bdB_{h}^{n-1}\right) + \left(f^{n+1}_h  - f^{n}_h \right)\bdB_{h}^{n-1}\right|}{\left(1+\sigma\left|\bdB_{h}^{n}\right|^{2}\right)\left(1+\sigma\left|\bdB_{h}^{n-1}\right|^{2}\right)} \cdot\left|\nabla \times \bdB_{h}^{n}\right| d \bdx \\
			& +  \sum_{n=1}^{K-1} \int_{\Omega}  \frac{\left|f^{n+1}_h \bdB_{h}^{n} \sigma \left| \bdB_{h}^{n-1} \right|^{2} - f^{n}_h \bdB_{h}^{n-1} \sigma\left|\bdB_{h}^{n}\right|^{2} \right|}{\left(1+\sigma\left|\bdB_{h}^{n}\right|^{2}\right)\left(1+\sigma\left|\bdB_{h}^{n-1}\right|^{2}\right)} \cdot\left|\nabla \times \bdB_{h}^{n}\right| d \bdx. 
		\end{aligned}
	\end{equation*}
	We next use Young's inequality again to obtain
	\begin{equation*}
		\begin{aligned}
			& \left|\sum_{n=1}^{K} \int_{\Omega} \frac{f^{n}_h \bdB_{h}^{n-1}}{1+\sigma\left|\bdB_{h}^{n-1}\right|^{2}}  \cdot \nabla \times\left(\bdB_{h}^{n}-\bdB_{h}^{n-1}\right) d \bdx\right|
			\\
			\leq & \frac{1}{8} \left\|\nabla \times \bdB_{h}^{K} \right\|_{\beta_h}^2
			+ C \left\|\bdB_{h}^{K-1}\right\|^2
			+ \frac{1}{8} \left\|\nabla \times \bdB_{h}^{0} \right\|_{\beta_h}^2
			+ C \left\|\bdB_{h}^{0}\right\|^2 \\
			& + \frac{1}{8}\tau \sum_{n=1}^{K-1} \left\|\partial_\tau \bdB_{h}^{n}\right\|^2  
			+ C \tau \sum_{n=1}^{K-1} \left\|\nabla \times \bdB_{h}^{n}\right\|_{\beta_h}^2
			+ C \sum_{n=1}^{K-1} \left\|\bdB_{h}^{n-1}\right\|^2  
			+ C \tau \sum_{n=1}^{K-1} \left\|\nabla \times \bdB_{h}^{n}\right\|_{\beta_h}^2\\
			& +  \sigma \sum_{n=1}^{K-1} \int_{\Omega}  \frac{\left|f^{n+1}_h  \left| \bdB_{h}^{n-1} \right|^{2} \left(\bdB_{h}^{n} -\bdB_{h}^{n-1}\right) 
				+ \left(f^{n+1}_h -f^{n}_h\right) \bdB_{h}^{n-1} \left|\bdB_{h}^{n-1}\right|^2 \right|}{\left(1+\sigma\left|\bdB_{h}^{n}\right|^{2}\right)\left(1+\sigma\left|\bdB_{h}^{n-1}\right|^{2}\right)} \cdot\left|\nabla \times \bdB_{h}^{n}\right| d \bdx \\
			& +  \sigma \sum_{n=1}^{K-1} \int_{\Omega}  \frac{\left| f^{n}_h \bdB_{h}^{n-1} \left(\bdB_{h}^{n-1} + \bdB_{h}^{n}\right)
				\left(\bdB_{h}^{n-1} - \bdB_{h}^{n}\right) \right|}{\left(1+\sigma\left|\bdB_{h}^{n}\right|^{2}\right)\left(1+\sigma\left|\bdB_{h}^{n-1}\right|^{2}\right)} \cdot\left|\nabla \times \bdB_{h}^{n}\right| d \bdx \\
		\end{aligned}
	\end{equation*}
	By using the conclusion \eqref{WD3}, we obtain
	\begin{equation}
		\begin{aligned}
			& \left|\sum_{n=1}^{K} \int_{\Omega} \frac{f^{n}_h \bdB_{h}^{n-1}}{1+\sigma\left|\bdB_{h}^{n-1}\right|^{2}}  \cdot \nabla \times\left(\bdB_{h}^{n}-\bdB_{h}^{n-1}\right) d \bdx\right|
			\\
			\leq & \frac{1}{8} \left\|\nabla \times \bdB_{h}^{K} \right\|_{\beta_h}^2
			+ \frac{1}{2}\tau \sum_{n=1}^{K-1} \left\|\partial_\tau \bdB_{h}^{n}\right\|^2
			+ C \left\|\bdB_{h}^{K-1}\right\|^2
			+ \frac{1}{8} \left\|\nabla \times \bdB_{h}^{0} \right\|_{\beta_h}^2
			+ C \left\|\bdB_{h}^{0}\right\|^2  \\
			& + C \tau \sum_{n=1}^{K-1} \left\|\nabla \times \bdB_{h}^{n}\right\|_{\beta_h}^2
			+ C \sum_{n=1}^{K-1} \left\|\bdB_{h}^{n-1}\right\|^2  \\
			\leq & \frac{1}{8} \left\|\nabla \times \bdB_{h}^{K} \right\|_{\beta_h}^2
			+ \frac{1}{2}\tau \sum_{n=1}^{K} \left\|\partial_\tau \bdB_{h}^{n}\right\|^2  
			+ C\left(\left\|\bdB_{h}^{0}\right\|^2 + \left\|\nabla \times \bdB_{h}^{0}\right\|_{\beta_h}^2\right).\\
		\end{aligned}
		\label{SL4}
	\end{equation}
	
	Similarly, simplifying the third term at the right hand side of \eqref{WD5} yields
	\begin{equation}
		\begin{aligned}
			& \left|R_{m} \sum_{n=1}^{K} \int_{\Omega}  \left(\bdu^n_h \times \bdB_{h}^{n-1} , \nabla  \times \partial_{\tau}\bdB_{h}^{n}\right) d \bdx\right|
			\\
			\leq & \frac{1}{8} \left\|\nabla \times \bdB_{h}^{K} \right\|_{\beta_h}^2
			+ \frac{1}{4}\tau \sum_{n=1}^{K} \left\|\partial_\tau \bdB_{h}^{n}\right\|^2  
			+ C\left(\left\|\bdB_{h}^{0}\right\|^2 + \left\|\nabla \times \bdB_{h}^{0}\right\|_{\beta_h}^2\right).\\
		\end{aligned}
		\label{SL5}
	\end{equation}
	
	Combining \eqref{SL4} and \eqref{SL5}, then substituting them into \eqref{WD5}, we can obtain
	\begin{equation*}
		\begin{aligned}
			&\frac{1}{4} \tau \sum_{n=1}^{K} \left\|\partial_{\tau} \bdB_{h}^{n}\right\|^{2}
			+  \frac{1}{4} \left\|\nabla \times \bdB_{h}^{K}\right\|_{\beta_h}^2 
			\leq  \frac{1}{2} \left\|\nabla \times \bdB_{h}^{0}\right\|_{\beta_h}^2
			+ C\left(\left\|\bdB_{h}^{0}\right\|^2 + \left\|\nabla \times \bdB_{h}^{0}\right\|_{\beta_h}^2\right),
		\end{aligned}
	\end{equation*}
	which implies \eqref{WD4} by the discrete Gronwall's inequality.
\end{proof}



\section{Numerical results}\label{sect5}
In this section, we shall demonstrate the rationality of the spherical interface dynamo model \eqref{Z1} and the efficiency of the numerical scheme \eqref{EFE1} by convergence test and time evolution simulation. 
The simulations in this paper are all implemented based on the parallel adaptive finite element program development platform PHG \cite{PHG}.
The settings of the simulation are similar with those in \cite{Zhang2003,Chan2008,Cheng2020}. 
For the spherical region $\Omega$, without special specification, the radius of four subregions $\Omega_i(i = 1,2,3,4)$ are taken as $r_1 = 1.5$, $r_2 = 1.875$, $r_3 = 2.5$, and $r_4 = 7.5$.
\subsection{Convergence test}

We perform convergence tests on the unit ball for simplifying the calculations. The radius of the four subregions is scaled down equally. 
And we take the exact solution $\bdB_T(\bdx,t)=(\bdB_x, \bdB_y, \bdB_z)$ which satisfies the divergence-free condition as shown below.
\begin{numcases}{}
	\bdB_x = y \cdot e^{-t / m} \cdot (x^2 - 2  x  z + y^2 + 3  z^2 - 1), \nonumber \\
	\bdB_y = z \cdot e^{-t / m} \cdot (3  x^2 - 2  x  y + y^2 + z^2 - 1), \nonumber \\
	\bdB_z = x \cdot e^{-t / m} \cdot (x^2 + 3  y^2 - 2  y  z + z^2 - 1). \nonumber
\end{numcases}
The remaining parameters are as follows,
\begin{equation*}
	\begin{aligned}
		&\beta_1 = \beta_2 = \beta_3 =\beta_4 = 1, \quad  R_m = 1, \quad  R_{\alpha} = 1, \quad  \sigma = 1,\\
		&f(\bdx, t) = x^2 + y^2 + z^2, \quad  
		\bdu(\bdx, t) =\left(x^2yz, \, y^2xz, \, z^2xy\right).
	\end{aligned}
\end{equation*}
Additional source term should be added into the dynamical system to make $\bdB_T(\bdx,t)$ solve the nonlinear system exactly.
\begin{table}[htbp]
	\centering
	\caption{The error of $\bdB_h$ and $\nabla \times \bdB_h$ with constant time step on $\left[0,\,1\right]$. }
	\label{tab:1}
	\begin{tabular}{c|c|c|c|c|c}
		\Xhline{1.5pt}
		\multicolumn{2}{c|}{Time-space dimensions} & \multicolumn{2}{c|}{Error-$\bdB_h$}  & \multicolumn{2}{c}{Error-$\nabla \times \bdB_h$} \\
		\Xhline{1.2pt}
		$h_{\max}$ & $\tau$ & $\left\|\bdB_h - \bdB_T\right\|_{L^2}$ & Rate & $\left\|\nabla \times \left(\bdB_h - \bdB_T\right)\right\|_{L^2}$ & Rate \\
		\hline
		0.584484 & 0.1 & 2.169e-02 & —— & 2.614e-01 & —— \\
		\hline
		0.292247 & 0.1 & 7.244e-03 & 1.582 & 1.366e-01 & 0.936 \\
		\hline
		0.189921 & 0.1 & 3.248e-03 & 1.861 & 7.023e-02 & 1.544 \\
		\Xhline{1.5pt}
	\end{tabular}
\end{table}
With the numerical scheme \eqref{EFE1}, we integrate from $t=0$ to $t = 1$ in each mesh.
We fix the time step $\tau = 0.1$ and $m = 100$ to investigate the influence of the spatial mesh size on the magnetic field error and present the results in Table~\ref{tab:1}.
From the table, we can see that the error decreases in a certain ratio as the grid is continuously refined.
Meanwhile, to better show the results, we show the magnetic field distribution of the exact solution as well as the numerical solution in Figure~\ref{fig:A1}. It can be seen that the exact and numerical solutions are in good agreement.
Further, we fix the spatial step $h_{\max} = 0.189921$ and take $m=1$ to investigate the influence of the time step size on the magnetic field error and present the results in Table~\ref{tab:2}. Similar results can be found from the data. 

With these two situations,  we can conclude that the exactly solution can be approximated very well by using the edge element method.
\begin{figure}[htbp]
	\centering
	\subfloat[$\bdB_T$]{\includegraphics[width=.46\linewidth]{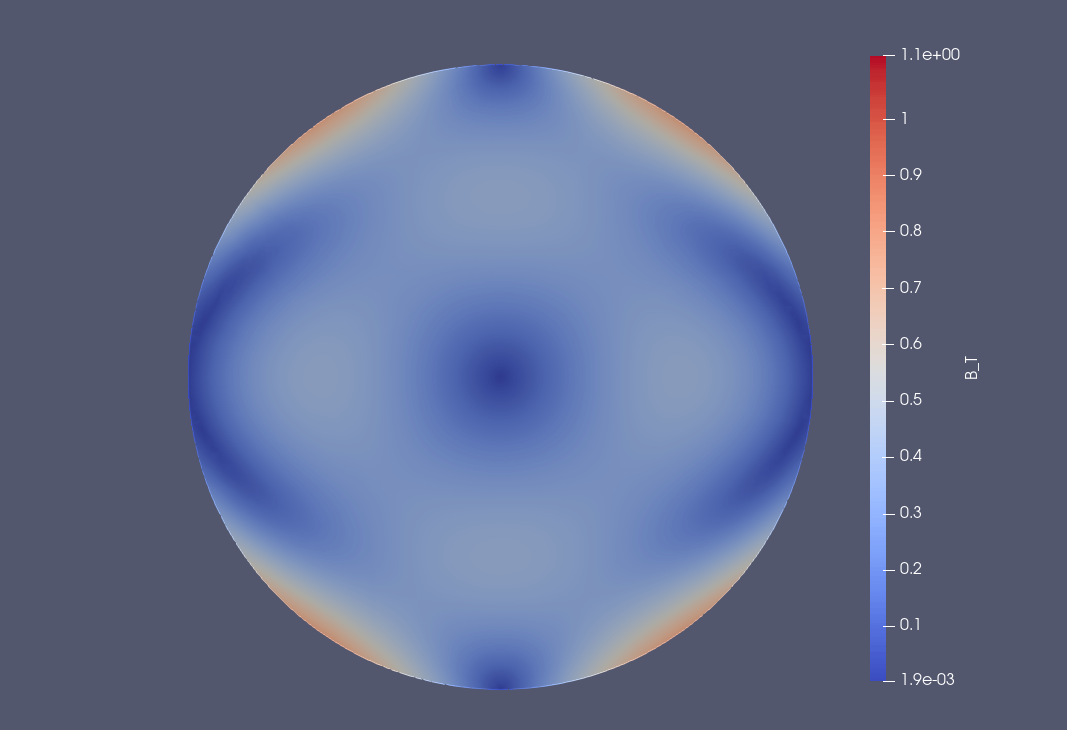}}\hspace{5pt}
	\subfloat[$\bdB_h$]{\includegraphics[width=.46\linewidth]{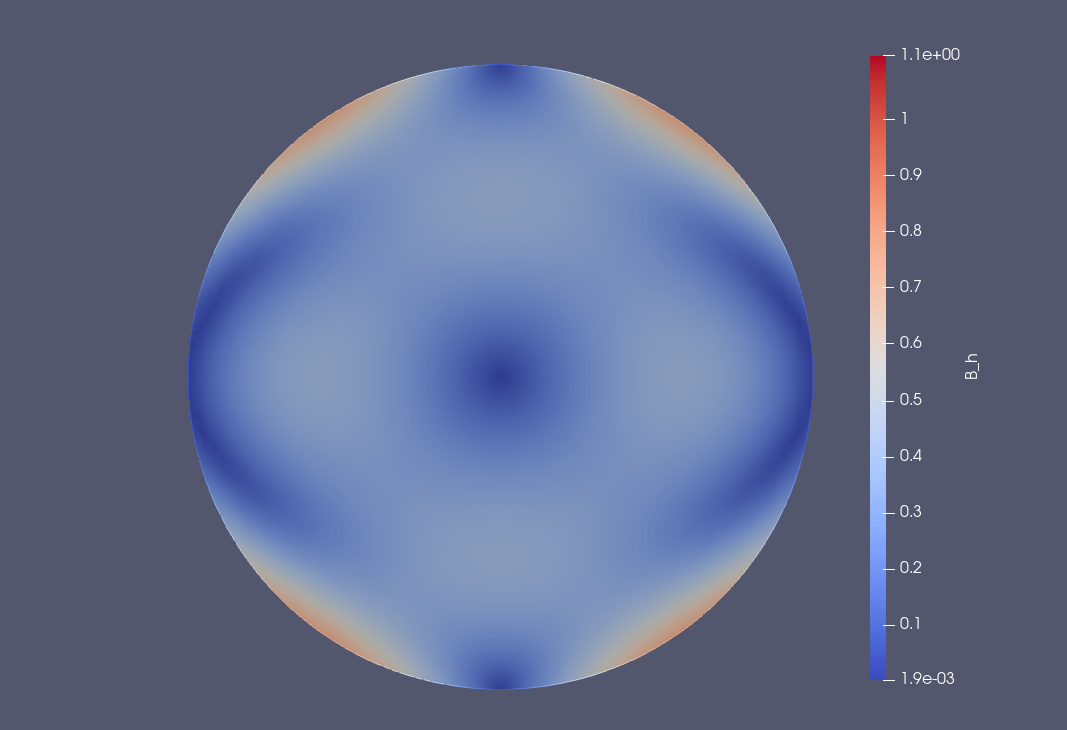}}
	\caption{The magnetic field distribution of the exact solution (left) and the numerical solution (right). \label{fig:A1}}
\end{figure}
\begin{table}[htbp]
	\centering
	\caption{The error of $\bdB_h$ and $\nabla \times \bdB_h$ with constant space step on $\left[0,\,1\right]$. }
	\label{tab:2}
	\begin{tabular}{c|c|c|c|c|c}
		\Xhline{1.5pt}
		\multicolumn{2}{c|}{Time-space dimensions} & \multicolumn{2}{c|}{Error-$\bdB$}  & \multicolumn{2}{c}{Error-$\nabla \times \bdB$} \\
		\Xhline{1.2pt}
		$h_{\max}$ & $\tau$ & $\left\|\bdB_h - \bdB_T\right\|_{L^2}$ & Rate & $\left\|\nabla \times \left(\bdB_h - \bdB_T\right)\right\|_{L^2}$ & Rate \\
		\hline
		0.189921 & 0.5 & 1.180e-01 & —— & 4.965e-01 & —— \\
		\hline
		0.189921 & 0.25 & 5.497e-02 & 1.102 & 2.319e-01 & 1.098 \\
		\hline
		0.189921 & 0.1 & 2.084e-02 & 1.059 & 9.105e-02 & 1.020 \\
		\Xhline{1.5pt}
	\end{tabular}
\end{table}

\subsection{Time evolution simulation}

In this part we perform the time evolution of the spherical interface dynamo system to simulate the mechanism of magnetic field generation on the Sun.

We take the corresponding magnetic diffusivity as $\beta_1 = 1$, $\beta_2 = 1$, $\beta_3 = 1$, and $\beta_4 = 150$.
Physically speaking, the $\alpha$-effect term of model \eqref{Z1} is used since it does not involve a strong nonlinear interaction between the flow and the Lorentz force \cite{Zhang2003}. Where, the positive parameter $\sigma$ is taken as 1 and the function $f$ as 
\begin{equation*}
	f\left(r,\theta,\phi\right)=\sin^2\theta\cos\theta\sin\left[\pi\frac{r-r_2}{r_3-r_2}\right],
\end{equation*}
where $\left(r,\theta,\phi\right)$ is the spherical polar coordinate.
For the other term, depending on the model settings, we choose the velocity field 
\begin{equation*}
	\bdu\left(r,\theta,\phi\right) = \left(u_r,u_\theta,u_\phi\right)=\left\{0,0,\Omega_t(\theta)r\sin\theta\sin\left[\pi\frac{r-r_1}{r_2-r_1}\right]\right\},
\end{equation*}
which acts only on the tachocline.  Where $\Omega_t(\theta)$ approximates the observed profile of the solar differential rotation \cite{Zhang2003}, which is chosen as the following three-term expression
\begin{equation*}
	\Omega_t(\theta) = 1 - 0.1264 \cos^2\theta - 0.1591\cos^4\theta.
\end{equation*}

\begin{figure}[!htbp]
	\centering
	\includegraphics[width=0.7\textwidth]{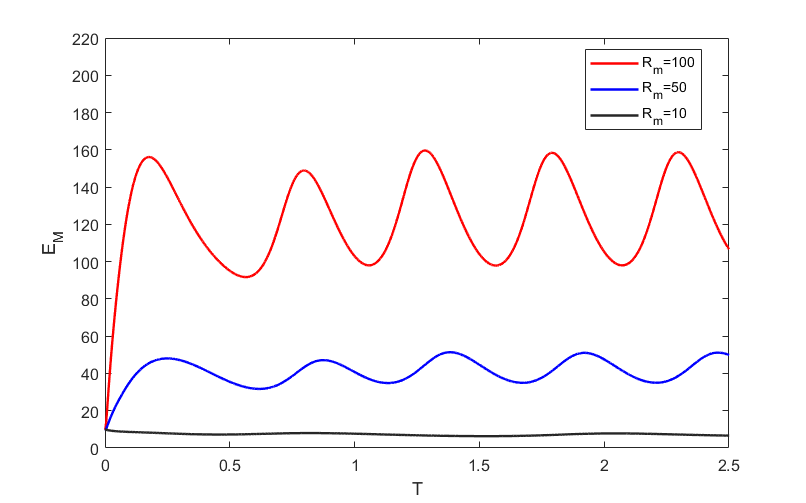}
	\caption{For $R_\alpha = 30$, $R_m = 10,50,100$, the magnetic field energy at different moments $E_M$.   \label{fig:En_M} }
\end{figure}

We take the initial value $\bdB_0=\left(\bdB_r,\bdB_\theta,\bdB_\phi\right)$ in $\Omega_1,\Omega_2$ and $\Omega_3$. Otherwise, we take $\bdB_0=0$. Where
\begin{numcases}{}
	\bdB_r = 2 \cos \theta r \left(r-r_3\right)^2 / r_3^2, \nonumber \\
	\bdB_\theta = - \sin \theta \left(3 r \left(r-r_3\right)^2 + 2 r^2 \left(r-r_3\right) \right) / r_3^2, \nonumber \\
	\bdB_\phi = 3 \cos \theta \sin \theta r^2 \left(r-r_3\right)^2 / r_3^2. \nonumber
\end{numcases}
\begin{figure}[!htbp]
	\centering
	\subfloat[t=0.0]{\includegraphics[width=.46\linewidth]{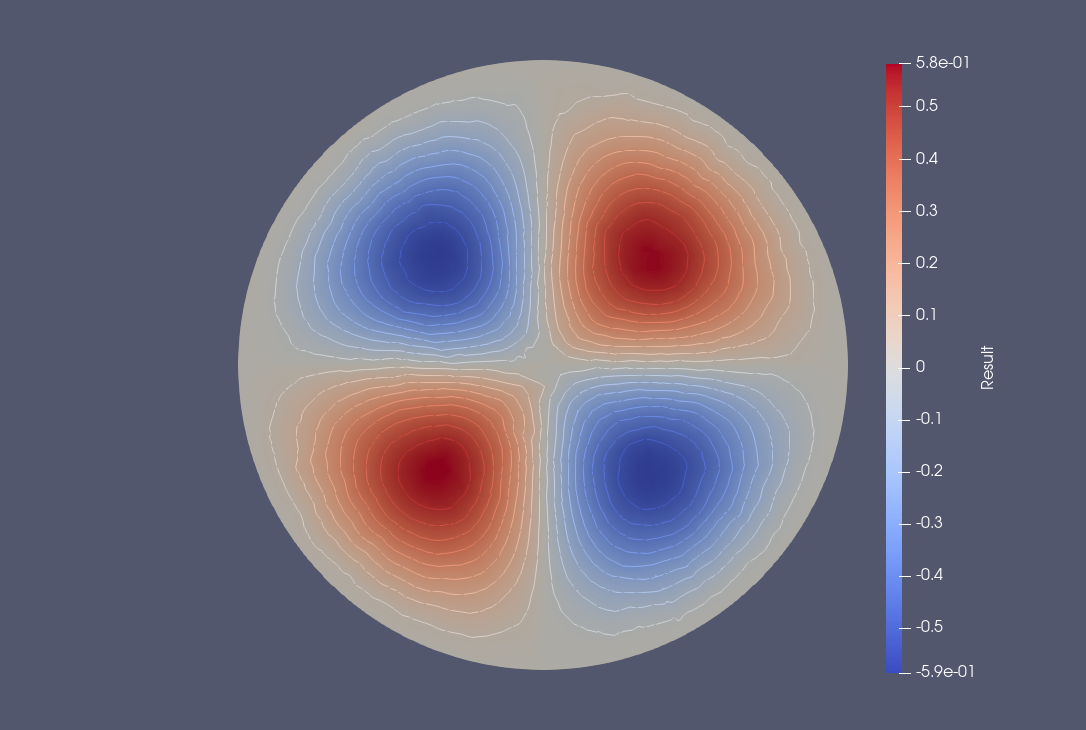}}\hspace{5pt}
	\subfloat[t=1.0]{\includegraphics[width=.46\linewidth]{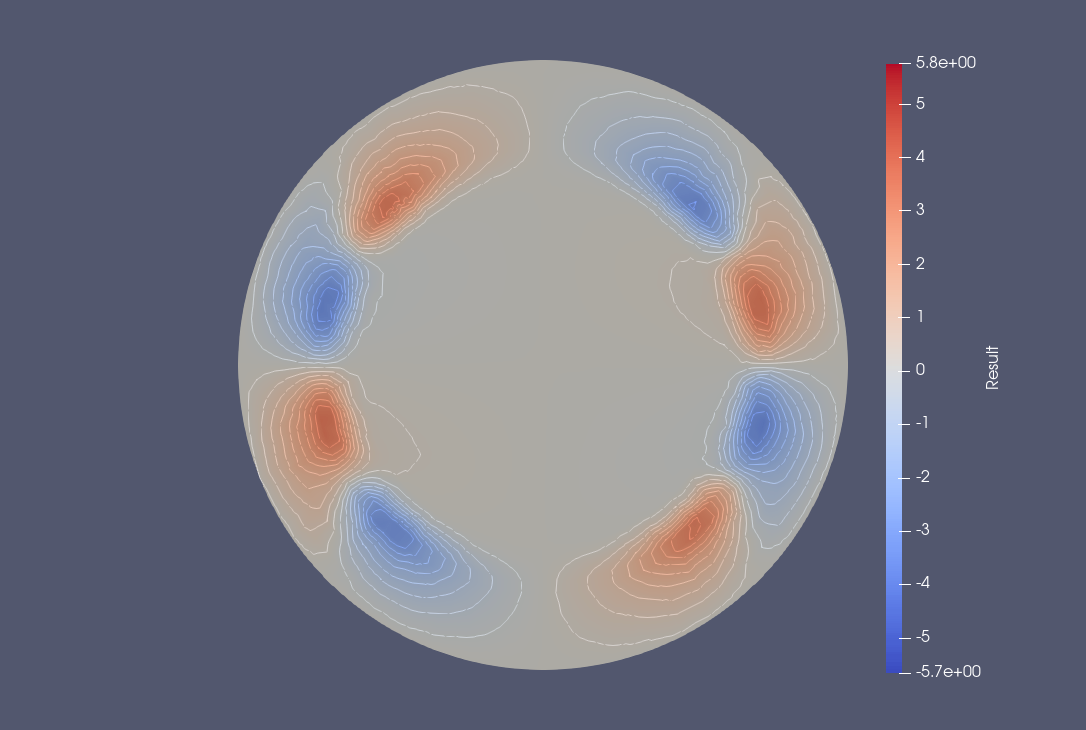}}\vspace{-2pt}
	\subfloat[t=1.3]{\includegraphics[width=.46\linewidth]{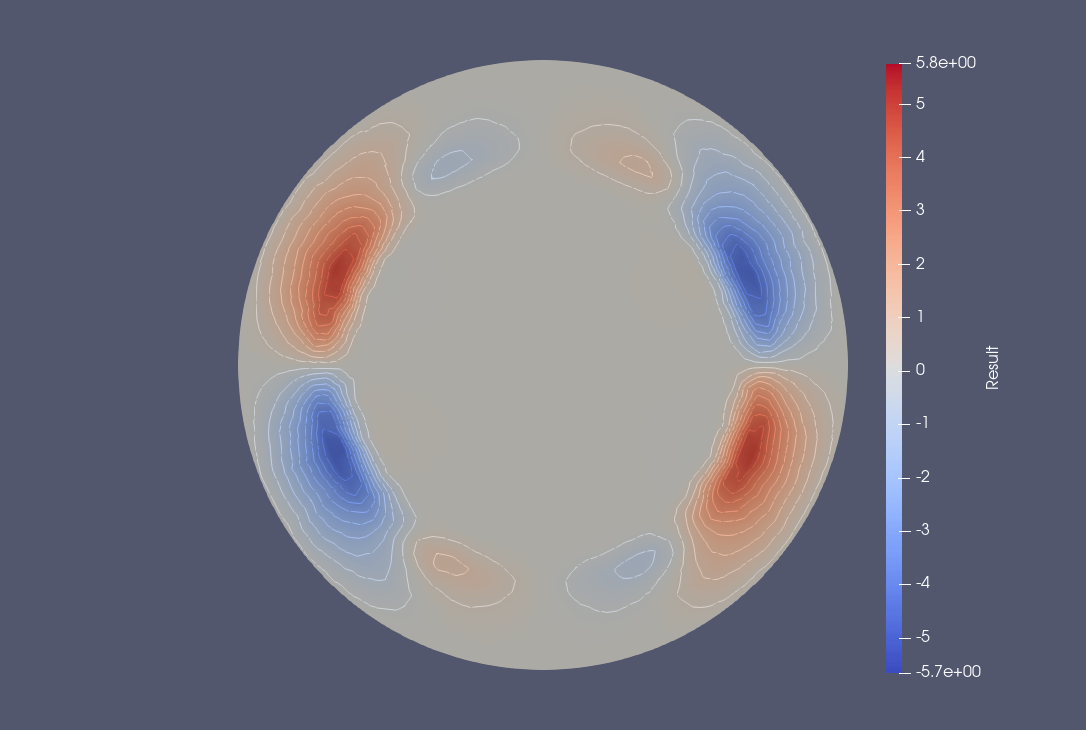}}\hspace{5pt}
	\subfloat[t=1.6]{\includegraphics[width=.46\linewidth]{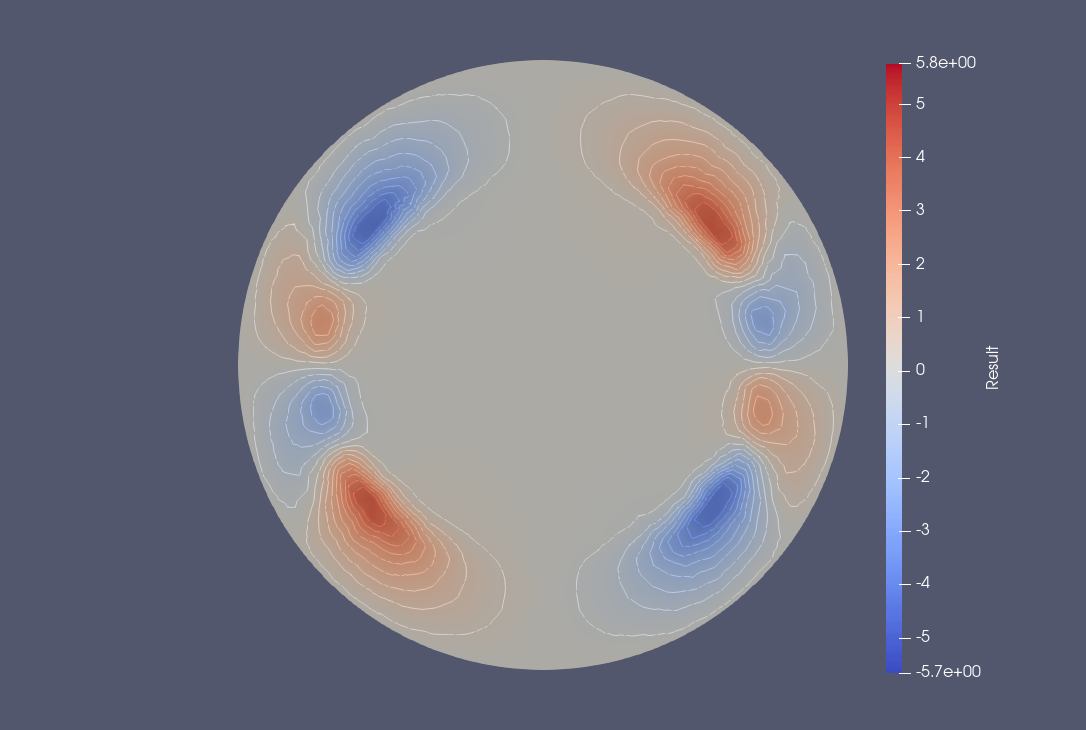}}\vspace{-2pt}
	\subfloat[t=1.9]{\includegraphics[width=.46\linewidth]{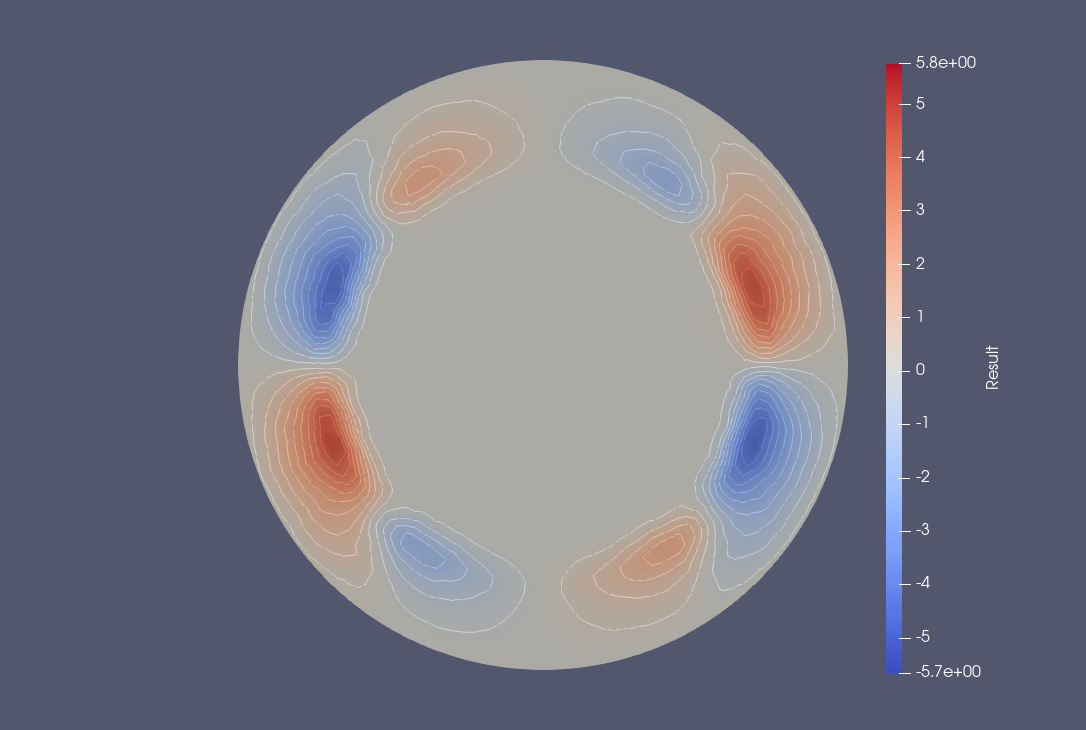}}\hspace{5pt}
	\subfloat[t=2.2]{\includegraphics[width=.46\linewidth]{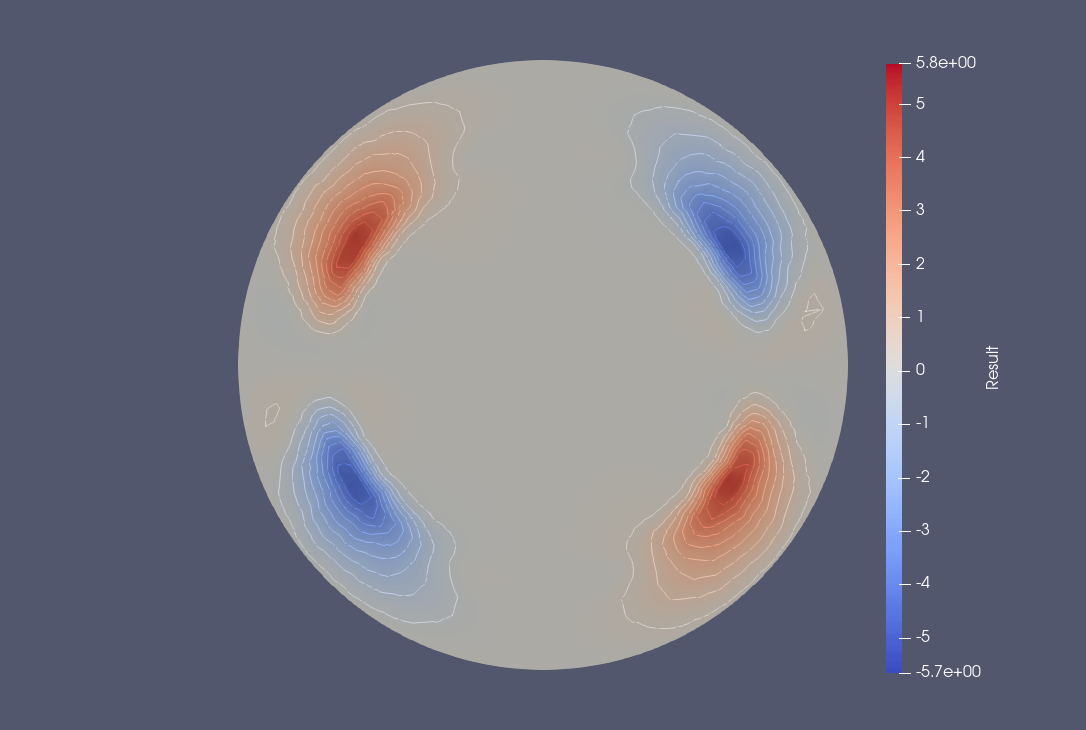}}
	\caption{Contour maps of the azimuthal field $\bdB_\phi$ in the $y-z$ plane at different moments $t$. The magnetic Reynolds number $R_\alpha = 30$, $R_m = 100$ with magnetic diffusivity $\beta_1 = 1$, $\beta_2 = 1$, $\beta_3 = 1$, and $\beta_4 = 150$. \label{fig:C1}}
\end{figure}
With the above settings, we simulate the generation mechanism of the solar interface dynamo by solving \eqref{EFE1}.
Before studying the variation of the magnetic field during the time evolution, we first investigate the variation of the magnetic field energy with different parameters to confirm the effectiveness of the model and the method. The magnetic field energy $E_M$ is defined as 
\begin{equation*}
	E_M=\int_V \left|\bdB\right|^2 d V,
\end{equation*}
where $V$ denotes the tachocline and the convection zone $(r_1 \leq r \leq r_3)$.

With $R_\alpha = 30$ fixed, we respectively present the energy variation curves for $R_m = $10, 50, and 100 in Figure~\ref{fig:En_M}. From the figure, we can see that the variation of magnetic field energy has a similar trend for different $R_m$. In the initial state, the magnetic field energy fluctuates irregularly with time, but as time advances, the magnetic field energy tends to stabilize with a periodic variation.
The difference is that the larger $R_m$ is, the larger the magnetic field energy is at the same moment.
Moreover, when the system reaches stability, the amplitude of the magnetic field energy curve increases with the increasing of $R_m$, while the wavelength decreases.
The properties are consistent with those reported in \cite{Zhang2003,Chan2006,Cheng2020}, which also confirms the effectiveness of the model and method.

Then in Figure~\ref{fig:C1}, we show  the time evolution of the interface dynamo system by plotting the azimuthal field $\bdB_\phi$ in the $y-z$ plane when $R_m = 100$. Here the $y-z$ plane means the union of the two half planes with azimuthal angles $\phi=\pi/2$ and $\phi=3\pi/2$ in spherical coordinates. 
From the contour maps at different moments, it can be seen that the magnetic field generated by the dynamo system is mainly concentrated near the interface between the tachocline and the convection zone, and it changes periodically with respect to time $t$.
The characteristics reflected by the above results are consistent with the existing work \cite{Zhang2003, Chan2006, Cheng2020}. 

\section{Concluding remarks}\label{sect6}
We remodel the spherical interface dynamo problem with quasi-vacuum boundary condition, which is very suitable to be discretized with the edge element method. Based on the new model, we provide an efficient way to solve the interface dynamo system. We first demonstrate the well-posedness of the continuous system by the fixed point theorem.
Then, we discretize the problem using the edge element method and present the stability of the corresponding discretization scheme.
Finally, we give some numerical examples to show that the numerical scheme is convergent and the simulation results are consistent with the existing results in the literature. We believe that the model and the method studied in this paper shall provide a useful tool for investigating celestial magnetic fields.

\bibliographystyle{plain}
\bibliography{reference}
	
\end{document}